\documentclass[12pt]{amsart}
\usepackage{amsmath,amssymb,amsfonts,amsthm,amsopn,dsfont}
\usepackage{graphics}

\usepackage{amscd,amsxtra}
\usepackage{latexsym}

\newtheorem{theorem}{Theorem}[section]
\newtheorem{corollary}[theorem]{Corollary}

\newtheorem{lemma}[theorem]{Lemma}
\numberwithin{equation}{section}

\newtheorem{proposition}[theorem]{Proposition}

\newcommand{\M}[2]{\mathcal{M}^{#1}_{#2}}
\renewcommand{\S}{\mathcal{S}}

\newcommand{\nm}[2]{\|{#1}\|_{#2}}

\newcommand{\so}[2]{\mathcal{H}_{#1}^{#2}}

\newcommand{\biggparen}[1]{\biggl(#1\biggr)}

\newcommand{\ip}[2]{\langle#1,#2\rangle}

\newcommand{\beqa}{\begin{eqnarray*}}
\newcommand{\eeqa}{\end{eqnarray*}}

\DeclareMathOperator*{\supp}{supp}

\newcommand{\field}[1]{\mathbb{#1}}
\newcommand{\bR}{\field{R}}        
\newcommand{\bZ}{\field{Z}}        
\def\la{\lambda}

\def\eps{\epsilon}

 \def\cF{\mathcal{F}}              
 \def\cS{\mathcal{S}}

 \def\cC{\mathcal{C}}

\def\a{\aleph}

\def\rd{\bR^d}

\def\zd{\bZ^d}

\def\intrd{\int_{\rd}}

\def\R{\right)}

\def\<{\left<}
\def\>{\right>}

\def\mv1{M_v^1}

\def\phas{(x,\o )}
\def\mn{(m,n)}
\def\mn'{(m',n')}

\def\o{\omega}
\def\a{\alpha}
\def\b{\beta}

\def\Z{\mathbb{Z}^{d}}

\def\R{\mathbb{R}}
\def\Ren{\mathbb{R}^d}

\def\Fur{\mathcal{F}}

\def\f{\varphi}

\def\Sn2{S_{2}(L^{2}(\Ren))}
\def\S1{S_{1}(L^{2}(\Ren))}
\def\sig00{\sigma_{0,0}}

\def\la{\langle}
\def\ra{\rangle}

\begin{document}

\begin{abstract} In this paper we give a sharp
 estimate on the norm of the scaling
operator $U_{\lambda}f(x)=f(\lambda x)$
acting on the weighted modulation
spaces $\M{p,q}{s,t}(\R^{d})$. In
particular, we recover and extend
recent results by
 Sugimoto and Tomita in the unweighted case
  \cite{sugimototomita}. As an application of our
  results, we
estimate the growth in time of
solutions of the wave and vibrating
plate equations, which is of interest
when considering the well posedeness of
the Cauchy problem for these equations.
Finally, we provide new embedding
results between modulation and Besov
spaces.
\end{abstract}

\title{Dilation properties for weighted modulation spaces}
\author{Elena Cordero and Kasso A.~Okoudjou$^{*}$}

\address{Department of Mathematics,  University of Torino,
Via Carlo Alberto 10, 10123
Torino, Italy}
\address{Department of Mathematics\\
University of Maryland\\
College Park, MD 20742, USA}

\thanks{$^*$Partially supported by ONR grant: N000140910324, and by a RASA from the Graduate School of UMCP}

\email{elena.cordero@unito.it}
\email{kasso@math.umd.edu}


\date{\today}
\maketitle

\section{Introduction}\label{intro}
The modulation spaces were introduced
by H.~Feichtinger \cite{Fei83}, by
imposing integrability conditions on
the short-time Fourier transform (STFT)
of tempered distributions. More
specifically, for $x, \omega \in
\R^{d}$, we let $M_\omega$ and $T_x$
denote the operators of modulation and
translation.  Then, the STFT of $f$
with respect to a nonzero window $g$ in
the Schwartz class is $$V_gf(x,
\omega)=\ip{f}{M_{\omega}T_{x}g}=\int_{\R^{2d}}
f(t)\overline{g(t-x)}e^{-2\pi i t\cdot
\omega}\, dt.$$ $V_{g}f(x, \omega)$ measures the frequency content of $f$ in a
neighborhood of $x$.

 For $s_1, s_2 \in \R$ and $1\leq p, q\leq\infty$, the
weighted modulation space
$\M{p,q}{s_{1}, s_{2}}(\R^{2d})$ is
defined to be the Banach space of all
tempered distributions $f$ such that
\begin{equation}\label{modspace}
\nm{f}{\M{p,q}{s_{1}, s_{2}}} =
\biggparen{\int_{\R^{2}}\biggparen{\int_{\R^{2}}|V_{g}f(x,
\omega)|^{p}\, v_{s_{1}}(x)^{p}\, dx}^{q/p}\, v_{s_{2}}(\omega)^{q}\, d\omega}^{1/q}
<\infty.
\end{equation}
 Here and in the sequel, we use the notation
$$v_{s}(x)=<x>^{s}=(1+|x|^{2})^{s/2}.$$
The  definition of modulation space is
independent of the choice of the window
$g$, in the sense that different window
functions yield equivalent
modulation-space norms. Furthermore,
the dual of a modulation space is also
a modulation space: if $p<\infty$, $
q<\infty$, $(\M{p, q}{s, t})'=\M{p',
q'}{-s, -t}$, where $p', q'$ denote the
dual exponents of $p$ and $q$,
respectively.

When both $s=t=0$, we will simply write
$\M{p, q}{}=\M{p, q}{0, 0}$. The weighted $L^{2}_{s}$ space  is
exactly $\M{2, 2}{s, 0}$, while an application of Plancherel's identity shows that
the Sobolev space
$\so{2}{s}$ coincides with $\M{2,2}{0, s}$.
For further properties and
uses of modulation spaces, see Gr\"ochenig's book \cite{book}, and we refer to
\cite{tri83} for equivalent definitions of the
modulation spaces for all $0<p,q\leq \infty$.

The modulation spaces appeared in recent years in various areas of mathematics
and engineering. Their relationship with other function spaces have
been investigated and resulted in embedding results of modulation spaces into other
function spaces such as the Besov and Sobolev spaces \cite{kasso04, sugimototomita,
toft04}. Sugimoto and Tomita \cite{sugimototomita} proved the optimality of certain
of the embeddings of modulation spaces into Besov space obtained  in \cite{kasso04,
toft04}. These results were obtained as consequence to optimal bounds of
$\|U_{\lambda}\|_{\M{p,q}{} \to \M{p,q}{}}$  \cite[Theorem 3.1]{sugimototomita},
where $U_{\lambda}f(\cdot)=f(\lambda \cdot)$ for $\lambda >0$.

The operator $U_\lambda$ has been investigated on many other function spaces
including the Besov spaces. For purpose of comparison with our results we include
the following results summarizing the behavior of $U_\lambda$ on the Besov spaces \cite[Proposition 3]{rusic}:

\begin{theorem}\label{dilbes}
 For $\lambda\in(0,\infty)$, $s\in\R$,
\begin{equation}\label{dilbesov}
C^{-1}\lambda^{-\frac{d}p }\min\{1,\lambda^s\}\|f\|_{B^{p,q}_s}\leq \|
f_\lambda\|_{B^{p,q}_s} \leq C\lambda^{-\frac{d}p
}\max\{1,\lambda^s\}\|f\|_{B^{p,q}_s}.
\end{equation}
\end{theorem}

\smallskip

The estimate on the norm of $U_{\lambda}$ on the (unweighted)
modulation spaces $\M{p,q}{}(\R^{d})$ was first obtained by
Sugimoto and Tomita \cite{sugimototomita}. In this paper, we shall
derive optimal lower and upper bounds for the operator
$U_{\lambda}$ on general modulation spaces $\M{p,q}{t,s}(\R^{d})$. More specifically, 
the boundedness of $U_{\lambda}$ on $\M{p,q}{t,s}$ is proved in Theorems \ref{xdil},
\ref{mainfreq} and \ref{mainboth}, and the optimal bounds on $\|U_{\lambda}\|_{\M{p,q}{t, s} \to \M{p,q}{t, s}}$ are established by Theorems \ref{sharp31} and \ref{sharp32}. We wish to point out that it is not trivial to
prove  sharp bounds  on the norm of the operator $U_{\lambda}$, as one has to construct examples of functions in the modulation spaces that achieve the desired optimal estimates.  We construct such examples by exploiting the properties of Gabor frames generated by the Gaussian window. 
It is likely that the functions that we construct can play
some role in other areas of analysis where the modulation are
used, e.,g., time-frequency analysis of pseudodifferential
operators and PDEs.

Interesting applications concern Strichartz estimates for
dispersive equations such as the wave
equation and the vibrating plate
equation on Wiener amalgam and
modulation spaces, where the time
parameter of the Fourier multiplier
symbol is considered as scaling factor.
We plan to investigate such applications in a
subsequent paper.

 Finally, we prove new embeddings between modulation
spaces and Besov spaces, generalizing some of the results of \cite{kasso04}.
Although strictly speaking this is not an application of the above dilation results,
it is clearly in the spirit of the main topic of the present paper,
so that we devote a short subsection to the problem.

\vskip0.1truecm

Our paper is organized as follows. In Section~\ref{prelim} we set up the notation
and prove some preliminary results needed to establish our theorems. In
Section~\ref{main} we prove the complete scaling of weighted modulation spaces. In
Section~\ref{sharpness} the sharpness of our results are proved, and in
Section~\ref{applic} we point out the applications of our main results.

Finally, we shall use the notations $A\lesssim B$ to mean that there exists a
constant $c>0$ such that $A\leq cB$, and $A\asymp B$ means that $A\lesssim B
\lesssim A$.

\section{Preliminary}\label{prelim}
We shall use the  set and index terminology of the paper \cite{sugimototomita}.
Namely, for $1\leq p\leq\infty$, let $p'$ be the
conjugate exponent of $p$
($1/p+1/p'=1$). For
$(1/p,1/q)\in [0,1]\times
[0,1]$, we define the subsets
$$ I_1=\max (1/p,1/p')\leq 1/q,\quad\quad I_1^*=\min (1/p,1/p')\geq 1/q,
$$
$$ I_2=\max (1/q,1/2)\leq 1/p',\quad\quad I_2^*=\min (1/q,1/2)\geq  1/p',
$$
$$ I_3=\max (1/q,1/2)\leq 1/p,\quad\quad I_3^*=\min (1/q,1/2)\geq
1/p.
$$
These sets are displayed in Figure 1:
\vspace{1.2cm}
 \begin{center}
           \includegraphics{figSchr1.1}
            \\
           $ $
\end{center}
 \begin{center}{\quad \quad\quad\quad \quad $0<\lambda\leq 1$\hfill   $\lambda\geq
1$}\quad\quad\quad\quad\quad\quad\quad\quad\quad
           \end{center}
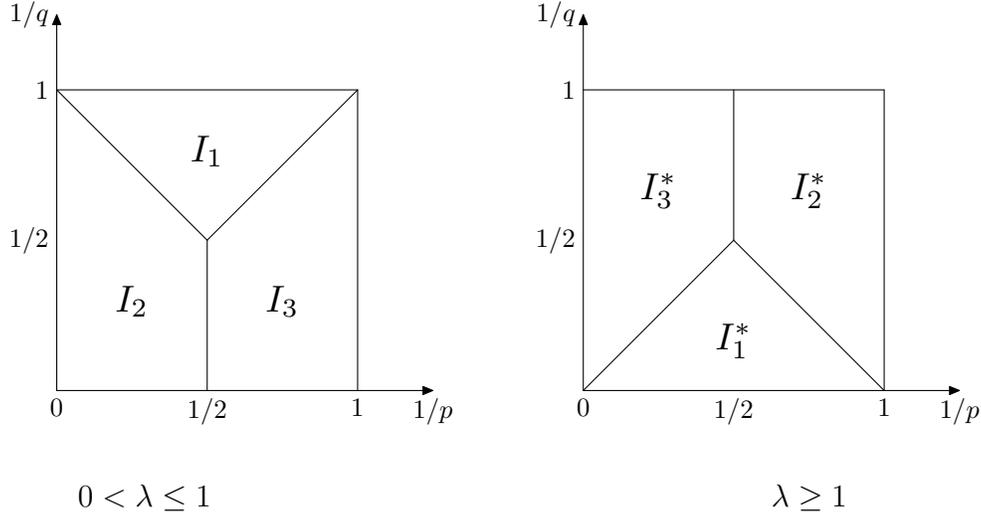
           \begin{center}{ Figure 1. The index sets. }
           \end{center}
  \vspace{1.2cm}

We introduce the indices:
$$ \mu_1(p,q)=\begin{cases}-1/p &  \quad {\mbox{if}} \quad (1/p,1/q)\in  I_1^*,\\
 1/q-1 &   \quad {\mbox{if}}  \quad (1/p,1/q)\in  I_2^*,\\
 -2/p +1/q&  \quad  {\mbox{if}}  \quad (1/p,1/q)\in  I_3^*,\\
 \end{cases}
 $$
and
$$ \mu_2(p,q)=\begin{cases}-1/p &  \quad {\mbox{if}} \quad (1/p,1/q)\in  I_1,\\
 1/q-1 &   \quad {\mbox{if}}  \quad (1/p,1/q)\in  I_2,\\
 -2/p +1/q&  \quad  {\mbox{if}}  \quad (1/p,1/q)\in  I_3.\\
 \end{cases}
 $$
 Next, we prove a lemma that will be used throughout this paper, and which allows us
to investigate the action of $U_{\lambda}$ only on $\cS(\rd)$.

\begin{lemma}\label{Fabio} Let $m$ be a polynomial growing weight function, A be a
linear continuous
operator from $\cS'(\rd)$ to $\cS'(\rd)$. Assume that
\begin{equation}\label{extension}
\|Af\|_{\M{p,q}{m}}\leq
C\|f\|_{\M{p,q}{m}},\quad \mbox{
for\,all}\ f\in\cS(\rd).
\end{equation}
Then
\begin{equation}\label{extension2}
\|Af\|_{\M{p,q}{m}}\leq C\|f\|_{\M{p,
q}{m}},\quad \mbox{ for\,all}\
f\in\M{p, q}{m}(\rd).
\end{equation}
\end{lemma}
\begin{proof}
The conclusion is clear if
$p,q<\infty$, because in that
case $\cS(\R^{d})$ is dense
in $\M{p,q}{m}(\R^{d})$.\par
Consider now the case
$p=\infty$ or $q=\infty$. For
any given $f \in \M{p,q}{m}$,
consider a sequence $f_n$ of
Schwartz functions, with
$f_n\to f$ in $\cS'(\R^{d})$,
and
\begin{equation}\label{controllo}\|f_n\|_{\M{p,q}{m}}\lesssim
\|f\|_{\M{p,q}{m}}
\end{equation}
(see the proof of Proposition
11.3.4 of \cite{book}). Since
$f_n$ tends to $f$ in
$\cS'(\rd)$, $Af_n$ tends to
$Af$ in $\cS'(\rd)$, and
$V_\f Af_n$ tends to  $V_\f
Af$  pointwise.
Hence, by Fatou's Lemma, the
assumption \eqref{extension}
and \eqref{controllo},
$$\|A f\|_{\M{p,q}{m}}\leq \liminf_{n\to\infty}
\|A f_n\|_{\M{p,q}{m}}\lesssim
\liminf_{n\to\infty}
\|f_n\|_{\M{p, q}{m}}\lesssim
\|f\|_{\M{p,q}{m}}.$$
\end{proof}

We shall also make use of the following characterization of the modulation spaces by
Gabor frames generated by the Gaussian function, which will be denoted through the
paper by  $\varphi(x)=e^{-\pi |x|^{2}}, x \in \R^d.$
Recall that for $0< a<1$, the family, $$\mathcal{G}(\varphi, a,
1)=\{\varphi_{k, \ell}(\cdot)=M_{\ell}T_{ak}\varphi=e^{2\pi i \ell \cdot}\varphi(\cdot - ak),
k, \ell \in \Z \}$$ is a Gabor frame for $L^{2}(\R^{d})$ if and only if there exist
$0<A\leq B < \infty$ such that for all $f \in L^2$ we have

\begin{equation}\label{frameineq}
A\|f\|_{L^{2}}^{2}\leq \sum_{k, \ell \in \Z}|\ip{f}{\varphi_{k,\ell}}|^{2}\leq
B\|f\|_{L^{2}}^{2}.
\end{equation} Moreover, there exists a dual function $\tilde{\varphi}\in \cS$ such
that $\mathcal{G}(\tilde{\varphi}, a, 1)$ is also a frame for $L^2$ and every $f \in
L^2$ can be written as
\begin{equation}\label{gabrecons}
f=\sum_{k, \ell \in \Z}\ip{f}{\tilde{\varphi}_{k,\ell}}\varphi_{k,\ell}= \sum_{k, \ell \in
\Z}\ip{f}{\varphi_{k,\ell}}\tilde{\varphi}_{k,\ell}.
\end{equation}

It is easy to see from the isometry of the Fourier transform on $L^2$ and the fact
that $\widehat{M_{\ell}T_{ak}\varphi}=T_{\ell}M_{-ak}\hat{\varphi}=e^{2\pi i a
k \ell}M_{-ak}T_{\ell}\varphi$, that $\mathcal{G}(\varphi, 1, a)$ is a Gabor frame whenever
$\mathcal{G}(\varphi, a, 1)$ is. The characterization of the modulation spaces by
Gabor frame is summarized in the following proposition. We refer to \cite[Chapter
9]{book} for a detail treatment of Gabor frames in the context of the modulation
spaces. In particular, the next result is proved in \cite[Theorem 7.5.3]{book} and
describe precisely when the Gaussian function generates a Gabor frame on $L^2$.

\begin{proposition}\label{gabframe}
$\mathcal{G}(\varphi, a, 1)$ is a Gabor frame for $L^2$ if and only if $0<a<1$.
In this case, $\mathcal{G}(\varphi, a, 1)$  is also a Banach frame for
$\M{p,q}{t,s}$ for all $1\leq p, q \leq \infty$, and $s, t \in \R$. Moreover, $f \in
\M{p,q}{t,s}$ if and only if there exists a sequence $\{c_{k,\ell}\}_{k, \ell \in \Z}\in
\ell^{p,q}_{t,s}(\Z\times \Z)$  such that
$
f=\sum_{k, \ell \in \Z}c_{k,\ell}\varphi_{k,\ell}$ with convergence in the modulation space norm. In addition,
$$\|f\|_{\M{p,q}{t,s}}\asymp
\|c\|_{\ell^{p,q}_{t,s}}:=\bigg(\bigg(\sum_{k\in
\Z}|c_{k,\ell}|^{p}v_{t}(k)^{p}\bigg)^{q/p}v_{s}(\ell)^{q}\bigg)^{1/q}.$$
\end{proposition}

\section{Dilation properties of weighted modulation spaces}\label{main}

We first consider the polynomial weights in the time variables $v_{t}(x)=\la
x\ra^t=(1+|x|^2)^{t/2}$, $t\in\R$.

\begin{theorem} \label{xdil}
Let $1\leq p,q \leq\infty$, $t\in\R$. Then the following are true:

\noindent (1) There exists a constant $C>0$ such that $\forall f \in
\M{p,q}{t,0},\,\lambda\geq 1$,
\begin{equation}\label{stima1}
C^{-1}\, \lambda^{d\mu_2(p,q)}\min\{1,\lambda^{-t}\}\,\|f\|_{\M{p,q}{t,0}}\leq \|
f_\lambda\|_{\M{p,q}{t,0}}\leq C
\lambda^{d\mu_1(p,q)}\max\{1,\lambda^{-t}\}\,\|f\|_{\M{p,q}{t,0}}.
\end{equation}

\noindent  (2) There exists a constant $C>0$ such that $\forall f \in
\M{p,q}{t,0},\,0<\lambda\leq 1$,
\begin{equation}\label{stima2}
 C^{-1}\, \lambda^{d\mu_1(p,q)}\min\{1,\lambda^{-t}\}\,\|f\|_{\M{p,q}{t,0}}\leq \|
f_\lambda\|_{\M{p,q}{t,0}}\leq  C
\lambda^{d\mu_2(p,q)}\max\{1,\lambda^{-t}\}\,\|f\|_{\M{p,q}{t,0}}.
\end{equation}
\end{theorem}

\begin{proof} We shall only prove the upper halves of each of the
estimates~\eqref{stima1} and~\eqref{stima2}. The lower halves will follow from the
fact that $0< \lambda \leq 1$ if and only if $1/\lambda \geq 1$ and
$f=U_{\lambda}U_{1/\lambda}f=U_{1/\lambda}U_{\lambda}f$.

We first consider the case $\lambda \geq 1$. Recall the definition of the dilation
operator $U_\lambda$ given by $U_\lambda f(x)=f(\lambda x)$.
Since the mapping $f\mapsto \la \cdot\ra^t f$ is an homeomorphism from
$\M{p,q}{t_0,s}$ to $\M{p,q}{t_0-t,s}$, $t,t_0,s\in \R$, see, e.g., \cite[Corollary
2.3]{Toftweight}, we have:
$$\|U_\lambda f\|_{\M{p,q}{t,0}}\asymp\|\la\cdot\ra^t U_\lambda f\|_{\M{p,q}{}}.
$$
Using $\la\cdot\ra^t U_\lambda f=U_\lambda(\la\lambda^{-1}\cdot\ra^t ) f$ and the
dilation properties for unweighted modulation spaces in \cite[Theorem
3.1]{sugimototomita}, we obtain
$$\|U_\lambda(\la\lambda^{-1}\cdot\ra^t  f)
\|_{\M{p,q}{}}\leq C \lambda^{d \mu_1(p,q)}
\|\la\lambda^{-1}\cdot\ra^t  f\|_{\M{p,q}{}}\asymp
 \lambda^{d \mu_1(p,q)} \|
  \la\cdot\ra^{-t}\la\lambda^{-1}\cdot\ra^t
  (\la\cdot\ra^{t}f)\|_{\M{p,q}{}}.
$$

Hence,
it remains to prove that the pseudodifferential operator with symbol
$g^{(t,\lambda)}(x):=\la
x\ra^{-t}\la\lambda^{-1}x\ra^{t}$
is bounded on $\M{p,q}{}$, and that its
norm is bounded above by
$\max\{1,\lambda^{-t}\}$.
By \cite[Theorem
14.5.2]{book}, this will follow once we prove that
$\|g^{(t,\lambda)}(x)\|_{\M{\infty,1}{}}
\lesssim\max\{1,\lambda^{-t}\}$.
To see this, observe first
that
\begin{equation}\label{aggiunta1}|g^{(t,\lambda)}(x)|\lesssim
\max\{1,\lambda^{-t}\},\quad \forall x \in \rd.
\end{equation} Indeed, let $
v^{(t,\lambda)}(x)= \la
\lambda^{-1} x\ra^t$.
Consider the case $t\geq 0$.
Since $\lambda\geq 1$, we
have $\lambda^{-1}|x|\leq
|x|$ and
$v^{(t,\lambda)}(x)\leq \la
x\ra^t$.

Analogously, for $t<0$,
 we have
$v^{(t,\lambda)}(x)\leq
\lambda^{-t}\la x\ra^t$. Consequently, we get the desired
estimates
\eqref{aggiunta1}.

Using the
inclusion
$\cC^{d+1}(\rd)\hookrightarrow
\M{\infty,1}{}(\rd)$ we have
\[
\|g^{(t,\lambda)}(x)\|_{\M{\infty,1}{}}\lesssim
\sup_{|\alpha|\leq
d+1}\sup_{x\in\rd}|\partial^\alpha
g^{(t,\lambda)}(x)|.
\]
By Leibniz' formula, the
estimate
$|\partial^\beta\langle
x\rangle^t|\lesssim \langle
x\rangle^{t-|\beta|}$ and
\eqref{aggiunta1} we see that
this last expression is
estimated by
$\max\{1,\lambda^{-t}\}$.

This concludes the proof of the upper half of~\eqref{stima1}.

We now consider the case $0<\lambda \leq 1$. Observe that by \cite{sugimototomita}
we have
$$\|U_\lambda(\la\lambda^{-1}\cdot\ra^t  f)\|_{\M{p,q}{}}
\leq C \lambda^{d \mu_2(p,q)}\|\la\lambda^{-1}\cdot\ra^t
 f\|_{\M{p,q}{}}\asymp \lambda^{d \mu_2(p,q)} \|
  \la\cdot\ra^{-t}\la\lambda^{-1}\cdot\ra^t
  (\la\cdot\ra^{t}f)\|_{\M{p,q}{}}
.$$

Moreover, one easily shows that~\eqref{aggiunta1} still holds  using the same
arguments along with the fact that
$v^{(t,\lambda)}(x)=\lambda^{-t}(\lambda^2+|x|^2)^{t/2}\leq
\lambda^{-t}\la x\ra^t$ when $t\geq 0$. Similarly, $v^{(t,\lambda)}(x)\leq
\la x\ra^t$ when $t<0$. In addition,
$g^{(t,\lambda)}(x)=g^{(-t,\lambda^{-1})}(\lambda^{-1}x)$.
Hence, by the proof of~\eqref{stima1} and \cite[Theorem
3.1]{sugimototomita}, we see that
\[
\|g^{(t,\lambda)}\|_{\M{\infty,1}{}}\lesssim
\|g^{(-t,\lambda^{-1})}\|_{\M{\infty,1}{}}\lesssim\max
\{1,\lambda^{-t}\}.
\]
This establishes the upper half of~\eqref{stima2}.
\end{proof}

We now consider the polynomial weights in the frequency variables $v_{s}(\omega)=\la
\o\ra^s$,
$s\in\R$.
\begin{theorem}\label{mainfreq}
Let $1\leq p,q \leq\infty$, $s\in\R$. Then the following are true:\\
(1) There exists a constant $C>0$ such that $\forall f \in
\M{p,q}{0,s},\,\lambda\geq 1,$
\begin{equation}\label{bof1}
C^{-1}\,\lambda^{d\mu_2(p,q)}\min\{1,\lambda^s\}\,\|f\|_{\M{p,q}{0,s}}\leq  \|
f_\lambda\|_{\M{p,q}{0,s}}\leq   C
\lambda^{d\mu_1(p,q)}\max\{1,\lambda^s\}\,\|f\|_{\M{p,q}{0,s}}.
\end{equation}
 (2) There exists a constant $C>0$ such that $\forall f \in
\M{p,q}{0,s},\,0<\lambda\leq 1,$
\begin{equation}\label{bof2}
C^{-1}\,\lambda^{d\mu_1(p,q)}\min\{1,\lambda^s\}\,\|f\|_{\M{p,q}{0,s}}\leq\|
f_\lambda\|_{\M{p,q}{0,s}}\leq
C\lambda^{d\mu_2(p,q)}\max\{1,\lambda^s\}\,\|f\|_{\M{p,q}{0,s}}.
\end{equation}
\end{theorem}

\begin{proof} Here we use the fact that the mapping $f\mapsto \la D\ra^s f$ is an
homeomorphism from $\M{p,q}{t,s_0}$
 to $\M{p,q}{t,s_0-s}$, $t,s,s_0\in\R$ (see \cite[Corollary 2.3]{Toftweight}). The
rest of the proof uses similar  arguments as those in Theorem \ref{xdil}.
\end{proof}

The next result follows immediately by combining the last two theorems.
\begin{corollary} \label{xodil}
Let $1\leq p,q \leq\infty$, $t,s\in\R$. Then the following are true:\\
(1) There exists a constant $C>0$ such that $\forall f \in
\M{p,q}{t,s},\,\lambda\geq 1,$
\begin{align*}
C^{-1}\lambda^{d\mu_2(p,q)}\min\{1,\lambda^{-t}\}\min\{1,\lambda^{s}\}\,&\|f\|_{\M{p,q}{t,s}}
\leq \| f_\lambda\|_{\M{p,q}{t,s}}\\[1 \jot]
&  \leq  C
\lambda^{d\mu_1(p,q)}\max\{1,\lambda^{-t}\}\max\{1,\lambda^{s}\}\,\|f\|_{\M{p,q}{t,s}}.
\end{align*}
 (2) There exists a constant $C>0$ such that $\forall f \in
\M{p,q}{t,s},\,0<\lambda\leq 1,$
\begin{align*}
C^{-1}
\lambda^{d\mu_1(p,q)}\min\{1,\lambda^{-t}\}\min\{1,\lambda^{s}\}\,&\|f\|_{\M{p,q}{t,s}}
\leq \| f_\lambda\|_{\M{p,q}{t,s}}\\[1 \jot]
&   \leq C
\lambda^{d\mu_2(p,q)}\max\{1,\lambda^{-t}\}\max\{1,\lambda^{s}\}\,\|f\|_{\M{p,q}{t,s}}.
\end{align*}
\end{corollary}

The following result is an analogue of Corollary~\ref{xodil} for modulation spaces
defined by non-separable polynomial growing
weight function such as $v_s \phas:=\la \phas\ra^s=(1+|x|^2+|\o|^2)^{s/2}$. 
\begin{theorem}\label{mainboth}
Let $1\leq p,q \leq\infty$, $s\in\R$. Then the following are true:\\
(1) There exists a constant $C>0$ such that $\forall f \in
\M{p,q}{v_s},\,\lambda\geq 1,$
\begin{equation}
C^{-1} \lambda^{d\mu_2(p,q)}\min\{\lambda^{-s},\lambda^s\}\,\|f\|_{\M{p,q}{v_s}}\leq
\| f_\lambda\|_{\M{p,q}{v_s}}   \leq  C
\lambda^{d\mu_1(p,q)}\max\{\lambda^{-s},\lambda^s\}\,\|f\|_{\M{p,q}{v_s}}.
\label{boxf1}
\end{equation}
 (2) There exists a constant $C>0$ such that $\forall f \in
\M{p,q}{v_s},\,0<\lambda\leq 1,$
\begin{equation}
C^{-1} \lambda^{d\mu_1(p,q)}\min\{\lambda^{-s},\lambda^s\}\,\|f\|_{\M{p,q}{v_s}}
\leq \| f_\lambda\|_{\M{p,q}{v_s}}
  \leq  C \lambda^{d\mu_2(p,q)}\max\{\lambda^{-s},\lambda^s\}\,\|f\|_{\M{p,q}{v_s}}.
\label{boxf2}
\end{equation}
\end{theorem}

\begin{proof}

We  assume  $s\geq 0$.   A duality argument can be used to complete
the proof when $s<0$. (Notice, this duality argument will be given explicitly below
in the proof of the sharpness of Theorem~\ref{xdil} in the case $(1/p,1/q)\in I_2$,
$t\geq
0$).

Moreover, since the result has been
proved in \cite[Theorem
3.1]{sugimototomita} for $s=0$, one can
use interpolation arguments along with
Lemma~\ref{Fabio} to reduce the proof
when $s$ is an even integer.

The mapping $f\mapsto
\la x,D\ra^s f$ is an
homeomorphism from
$\M{p,q}{v_s}$
 to $\M{p,q}{}$, $s\in\R$
 (see \cite[Theorem 2.2]{Toftweight}). Hence
 \begin{align*} \|f_\lambda\|_{\M{p,q}{v_s}}&\asymp
 \|\la x,D\ra^s f_\lambda\|_{\M{p,q}{}}\\
 &=\|U_\lambda (\la \lambda^{-1} x,\lambda D\ra ^sf)
 \|_{\M{p,q}{}}\\
 &\leq C \begin{cases}
\lambda^{d\mu_1(p,q)}\|\la \lambda^{-1} x,\lambda
D\ra^s f \|_{\M{p,q}{}}, \quad \lambda\geq1\\
\lambda^{d\mu_2(p,q)}\|\la \lambda^{-1} x,\lambda
D\ra^s f \|_{\M{p,q}{}}, \quad 0<\lambda\leq1,
\end{cases}
 \end{align*}
where in the last inequality
we used again the dilation
properties for unweighted
modulation spaces of
\cite[Theorem
3.1]{sugimototomita}.
Therefore, writing $f=\langle
x,D\rangle^{-s} \langle
x,D\rangle^{s}f$ we see that
it suffices to prove that the pseudodifferential
operator
\[
\la \lambda^{-1} x,\lambda
D\ra^s \langle
x,D\rangle^{-s}
\]
is bounded on $\M{p,q}{}$, and its norm is bounded above by
$\max\{1,\lambda^{-s}\}\max\{1,\lambda^s\}=\max\{\lambda^s,\lambda^{-s}\}$.
To this end, we observe that,
if $s$ is an even integer,
$\la \lambda^{-1} x,\lambda
D\ra^s$ is a finite sum of
operators of the form
$\lambda^{k}x^\alpha
D^\beta$, with $|k|\leq s$
and $|\alpha|+|\beta|\leq s$.
Now, Shubin's
pseudo-differential calculus
\cite{shubin} shows that the
operators $x^\alpha D^\beta
\langle x,D\rangle^{-s}$ have
bounded symbols, together
with all their derivatives,
so that they are bounded on
$\M{p,q}{}$. The proof is completed by taking into
account the additional factor
$\lambda^{k}$. 
\end{proof}

Finally, it is relatively straightforward to give optimal estimates for the dilation
operator $U_{\lambda}$ on the Wiener amalgam spaces $W(\cF
L^p_s,L^q_t)$. These spaces are images of modulation spaces under Fourier transform,
that is $\cF \M{p,q}{t,s}=W(\cF L^p_s,L^q_t)$. It is also worth noticing that the
indices $\mu_1$ and $\mu_2$ obey the following relations,
$$ \mu_1(p',q')=-1-\mu_2(p,q), \quad \mu_2(p',q')=-1-\mu_1(p,q)\, \,
\textrm{whenever}\, \, \tfrac{1}{p}+\tfrac{1}{p'}=\tfrac{1}{q}+\tfrac{1}{q'}=1.$$
 Using the above relations along with the definition of the Wiener amalgam spaces,
as well as the behavior of the Fourier transform under dilation, i.e.,
$\widehat{f_\lambda}=\lambda^{-d}(\hat{f})_{\frac1\lambda}$ and Corollary
\ref{xodil} we obtain the following result
\begin{proposition}\label{mainbothW}
Let $1\leq p,q \leq\infty$, $t,s\in\R$. Then the following are true:\\
(1) There exists a constant $C>0$ such that $\forall f \in W(\cF
L^p_s,L^q_t),\,\lambda\geq 1,$
\begin{align*}
C^{-1}
\lambda^{d\mu_2(p',q')}\min\{1,\lambda^{t}\}\min\{1,&\lambda^{-s}\}\,\|f\|_{W(\cF
L^p_s,L^q_t)} \leq \| f_\lambda\|_{W(\cF L^p_s,L^q_t)}\\[1 \jot]
&  \leq  C
\lambda^{d\mu_1(p',q')}\max\{1,\lambda^{t}\}\max\{1,\lambda^{-s}\}\,\|f\|_{W(\cF
L^p_s,L^q_t)}.
\end{align*}
 (2) There exists a constant $C>0$ such that $\forall f \in W(\cF
L^p_s,L^q_t),\,\lambda\leq 1,$
\begin{align*}
C^{-1}
\lambda^{d\mu_1(p',q')}\min\{1,\lambda^{t}\}\min\{1,&\lambda^{-s}\}\,\|f\|_{W(\cF
L^p_s,L^q_t)} \leq \| f_\lambda\|_{W(\cF L^p_s,L^q_t)}\\[1 \jot]
&   \leq
C\lambda^{d\mu_2(p',q')}\max\{1,\lambda^{t}\}\max\{1,\lambda^{-s}\}\,\|f\|_{W(\cF
L^p_s,L^q_t)}.
\end{align*}
\end{proposition}

\section{Sharpness of Theorems \ref{xdil} and \ref{mainfreq}.}\label{sharpness}
In this section we prove the sharpness of Theorems \ref{xdil} and \ref{mainfreq}.
The sharpness of Theorem \ref{mainboth} is proved by modifying the examples constructed in the next subsection. Therefore we omit it. But we first prove some preliminary lemmas in which we construct functions that achieve
 the optimal bound.

\subsection{ Preliminary Estimates}
The next two lemmas involve estimates for the modulation space
norms of various modifications of the Gaussian. Together with
Lemmas~\ref{istar1}--\ref{i3negative}, they provide examples of
functions that achieve the optimal bound under the dilation
operator on weighted modulation spaces with weight on the space
parameter. Similar constructions for weighted modulation spaces
with weight on the frequency parameter are contained in
Lemmas~\ref{i1}--\ref{inftynegfreq}. Finally, in
Lemma~\ref{altertf} we investigated the property of the dilation
operator on compactly supported functions. 

Recall that $\varphi(x)=e^{-\pi |x|^{2}}$ for $x \in \rd$, and that $\varphi_{\lambda}(x)=U_{\lambda}\varphi(x)=\varphi(\lambda x).$

\begin{lemma}\label{Gaussian} For $t, s\geq
0$, we have
\begin{equation}\label{zero}
\|\f_\lambda\|_{M^{p,q}_{t,0}}\asymp\lambda^{-\frac d p-t},\quad 0<\lambda\leq 1,
\end{equation}
\begin{equation}\label{infinity}
\|\f_\lambda\|_{M^{p,q}_{t,0}}\asymp\lambda^{- d\left(1-\frac1q\right)},\quad
\lambda\geq1,
\end{equation}

\begin{equation}\label{zerof}
\|\f_\lambda\|_{\M{p,q}{0,s}}\asymp \lambda^{-\frac d p},\quad 0<\lambda\leq1,
\end{equation}
and
\begin{equation}\label{infinityf}
 \|\f_\lambda\|_{\M{p,q}{0,s}}\asymp \lambda^{- d\left(1-\frac1q\right)+s},\quad
\lambda\geq 1.
\end{equation}

\end{lemma}
\begin{proof}
We shall only prove the first two estimates,  as the last two are proved similarly.
By some straightforward computations, (see, e.g., \cite[Lemma 1.5.2]{book})  we get
\begin{equation}\label{aggiunta.0}
|V_\f \f_\lambda \phas| =
(\lambda^2+1)^{-\frac d
2}e^{-\pi
\frac{\lambda^2}{\lambda^2+1}|x|^2}e^{-\pi
\frac{1}{\lambda^2+1}|\o|^2}.
\end{equation}
Hence
\begin{align*}\|\f_\lambda\|_{\M{p,q}{t,0}}&\asymp \|V_\f
\f_\lambda\|_{\M{p,q}{t,0}}\\
&=p^{-2p}q^{-2q}\lambda^{-\frac{d}p}(\lambda^2+1)^{\frac d2(\frac 1p +\frac 1
q-1)}\left( \intrd e^{-\pi |x|^2}\la
\frac{\sqrt{\lambda^2+1}}{\lambda\sqrt{p}}x\ra^{pt} \,d x\right)^{\frac 1p}.
\end{align*}
If $0< \lambda \leq 1$, then
$$\lambda^{-t}|x|^{t/2}\leq (\tfrac{\lambda^{2}+1}{\lambda^{2}}|x|^{2})^{t/2}\leq
\big(1 + \tfrac{\lambda^{2}+1}{\lambda^{2}}|x|^{2}\big)^{t/2}\leq 2 \lambda^{-t} (1+
|x|^{2})^{t/2}.$$ Thus, we have
$$ \lambda^{-t}\lesssim \left(\intrd e^{-\pi |x|^2}\la
\frac{\sqrt{\lambda^2+1}}{\lambda\sqrt{p}}x\ra^{pt}
\,d x\right)^{1/p}\lesssim \lambda^{-t}, \quad 0< \lambda\leq 1,$$
and the estimate \eqref{zero} follows.

Now, observe that, if $\lambda\geq 1$, then  $\la
\frac{\sqrt{\lambda^2+1}}{\lambda\sqrt{p}}x\ra \asymp \la x \ra$ and
\eqref{infinity} follows.
\end{proof}

\begin{lemma}\label{i2star} For $t\leq 0$, $\lambda\geq1$, consider the family of
functions
\begin{equation}\label{TransGaus}
f(x)=\lambda^{-t}\f(x-\lambda
e_1),\ \
e_1=(1,0,0,\ldots,0).
\end{equation}
Then there exists a constant $C>0$  such that
$\|f\|_{\M{p,q}{t,0}}\leq C $, uniformly with respect to $\lambda$. Moreover,
\begin{equation}\label{bo1}
\|f_\lambda\|_{\M{p,q}{t,0}}\gtrsim \lambda^{-t+d(\frac1q-1)},\quad
\forall\,\lambda\geq1.
\end{equation}
\end{lemma}
\begin{proof}
We have
\begin{align*}
\|f\|_{\M{p,q}{t,0}}&\asymp
\|V_\f f(x,\omega)\langle
x\rangle^t\|_{L^{p,q}}=\lambda^{-t}\|V_\f
\f(x-\lambda
e_1,\omega)\langle
x\rangle^t\|_{L^{p,q}}\\
&=\lambda^{-t}\|V_\f
\f(x,\omega)\langle x+\lambda
 e_1\rangle^t\|_{L^{p,q}}\lesssim\lambda^{-t}\lambda^t\|V_\f
\f\la x\ra^{-t}\|_{L^{p,q}}\lesssim1.
\end{align*}
 The last inequality follows from the fact that the weight $\la\cdot\ra^t$ is
$\la\cdot\ra^{-t}$-moderate which implies that
$\langle x+\lambda
 e_1\rangle^t\lesssim
 \lambda^t\langle
 x\rangle^{-t}$. This proves the first part of the Lemma. Let us now
 estimate
 $\|f_\lambda\|_{\M{p,q}{t,0}}$
 from below. We have
 \[
 f_\lambda(x)=\lambda^{-t}\f_\lambda(x-e_1).
 \]
 Hence, by arguing as above and using  \eqref{aggiunta.0},  we have
 \begin{align*}
 \|f_\lambda\|_{\M{p,q}{t,0}}&
 \asymp\lambda^{-t}\|V_\f
 \f_\lambda (x,\omega)\langle
 x+e_1\rangle^t\|_{L^{p,q}}\\
 &\gtrsim \lambda^{-t} \lambda^{d(\frac1q-1)} \Bigg(\int
 e^{-\pi p|x|^2} \langle
 x+e_1\rangle^{pt} dx\Bigg)^{\frac1p} \gtrsim
 \lambda^{-t+d(\frac1q-1)},
 \end{align*}
which concludes the proof.
\end{proof}

\begin{lemma}\label{istar1}
Let $1\leq p, q \leq \infty$,  $\epsilon>0$, $t\in\R$, and $\lambda>1 $.  Moreover, assume
that $(1/p,1/q)\in I^*_1$.

\noindent a) If $t\geq 0$, define
$$
f(x)=\sum_{\ell\neq 0}|\ell|^{-d/p-\epsilon}e^{2\pi i \lambda^{-1}\ell\cdot
x}\varphi(x)=\sum_{\ell\neq 0}|\ell|^{-d/p-\epsilon}M_{\lambda^{-1}\ell}\varphi(x),\quad
\mbox{in}\,\, \cS'(\rd).$$
Then there exists a constant $C>0$  such that
$\|f\|_{\M{p,q}{t,0}}\leq C $, uniformly with respect to $\lambda$. Moreover,

\begin{equation}\label{Gaugf}
\|f_{\lambda}\|_{\M{p,q}{t,0}}\gtrsim\lambda^{-d/p - \epsilon}, \qquad  \forall \,
\lambda > 1.
\end{equation}

\noindent b) If $t\leq 0$ define
$$
f(x)=\sum_{k\neq 0}|k|^{-d/p-\epsilon-t}\varphi(x-k)=\sum_{k\neq
0}|k|^{-d/p-\epsilon-t}T_{k}\varphi(x),\quad  \mbox{in}\,\, \cS'(\rd).$$  Then there exists a constant $C>0$  such that
$\|f\|_{\M{p,q}{t,0}}\leq C $, uniformly with respect to $\lambda$. Moreover,
\begin{equation}\label{Gaugffreq}
\|f_{\lambda}\|_{\M{p,q}{t,0}}\gtrsim  \lambda^{-d/p - \epsilon-t}, \qquad \forall \,
\lambda > 1.
\end{equation}
\end{lemma}

\begin{proof}We only prove part $a)$ as part $b)$ is obtained similarly.
We use Proposition~\ref{gabframe} to prove that $f$ defined in the lemma belongs to $\M{p,q}{t,0}$. 
Indeed, $\mathcal{G}(\varphi, 1, \lambda^{-1})$ is a Gabor frame, and the coefficients of $f$ in this frame are given by 
$c_{k,\ell}=\delta_{k,0}|\ell|^{-d/p-\epsilon}$ if $\ell\neq 0$ and $c_{0,0}=0.$
It is clear that 
$$\|c_{k,\ell}\|_{\ell^{p,q}_{t,0}}=\bigg(\sum_{\ell\in \zd}\bigg(\sum_{k\in
\zd}|c_{k,\ell}|^{p}\langle k
\rangle^{pt}\bigg)^{q/p}\bigg)^{1/q}=\bigg(\sum_{\ell\neq
0}|\ell|^{q(-d/p-\epsilon)}\bigg)^{1/q} < \infty,$$ because $q/p\geq 1$. Thus, $f \in \M{p,q}{t,0}$ with uniform norm (with respect to $\lambda$).

Given $\lambda >1,$ we have
$$\|f_{\lambda}\|_{\M{p,q}{t,0}}=\sup_{\|g\|_{\M{p',q'}{-t,
0}}=1}|\ip{f_{\lambda}}{g}|\geq \|\varphi\|_{\M{p',
q'}{-t,0}}^{-2}|\ip{f_{\lambda}}{\varphi}|.$$
Using relation~\eqref{aggiunta.0},
$$\ip{f_{\lambda}}{\varphi}=\sum_{\ell\neq
0}|\ell|^{-d/p-\epsilon}V_{\varphi}\varphi_{\lambda}(0, \ell)=\sum_{\ell\neq
0}|\ell|^{-d/p-\epsilon}(1+\lambda^2)^{-d/2}e^{\tfrac{-\pi
|\ell|^{2}}{\lambda^{2}+1}}.$$
Therefore, if $\lambda >1$,
\begin{align*}
\|f_{\lambda}\|_{\M{p,q}{t,0}} & \geq C\sum_{\ell\neq
0}|\ell|^{-d/p-\epsilon}(1+\lambda^2)^{-d/2}e^{\tfrac{-\pi
|\ell|^{2}}{\lambda^{2}+1}}\\
& \geq C \lambda^{-d}  \sum_{\ell\neq 0}|\ell|^{-d/p-\epsilon}e^{\tfrac{-\pi
|\ell|^{2}}{\lambda^{2}+1}}\\
& \geq C \lambda^{-d}  \sum_{0< |\ell| <
\lambda}|\ell|^{-d/p-\epsilon}e^{\tfrac{-\pi |\ell|^{2}}{\lambda^{2}+1}}\\
& \geq C \lambda^{-d} \lambda^{-d/p - \epsilon} \sum_{0< |\ell| < \lambda }e^{-\pi} \\
& \geq C \lambda^{-d} \lambda^{-d/p - \epsilon} e^{-\pi}
\lambda^{d}=C\lambda^{-d/p+\epsilon},
\end{align*} from which the proof follows.
\end{proof}

The next results extend \cite[Lemma 3.9]{sugimototomita} and \cite[Lemma
3.10 ]{sugimototomita}. 

\begin{lemma}\label{lemm1} Let $1\leq p,q \leq\infty$,   $t\geq 0$,
$\epsilon>0$. Suppose that $\psi\in \cS(\rd)$ satisfy $\supp\psi\subset [-1/2,1/2]^d$
and $\psi=1$ on $[-1/4,1/4]^d$.

\noindent a) If $1\leq q < \infty$, define
\begin{equation}\label{lem1ex1}
f(y)=\sum_{k\in\zd\setminus\{0\}} | k|^{-\frac{d}q-\epsilon-t} M_{k}T_k
\psi(y),\quad\mbox{in}\,\,\cS'(\rd).
\end{equation}
Then, $f\in \M{p,q}{t,0}(\rd)$ and
\begin{equation}\label{est1}
\|f_\lambda\|_{\M{p,q}{t,0}}\gtrsim\,
\lambda^{-d(\frac2p-\frac1q)+\epsilon-t},\,\quad\forall\,\,0<\lambda\leq
1.
\end{equation}

\noindent b)  If $q=\infty$, let
\begin{equation}\label{es21}
f(y)=\sum_{k\not=0} |k|^{-t}M_{k}T_k
\psi(y),\quad\mbox{in}\,\,\cS'(\rd).
\end{equation}
Then $f\in \M{p,\infty}{t,0}$ and
\begin{equation}\label{est2}\|f_\lambda\|_{\mathcal{M}^{p,\infty}_{t,0}}\gtrsim
\lambda^{-
\frac{2d}p-t},\quad\forall\,\, 0<\lambda\leq1.
\end{equation}

\end{lemma}

\begin{proof}
We only prove part $a)$, i.e., the case  $1\leq q<\infty$ as the case $q=\infty$ is
proved in a similar fashion.

Let  $g \in \cS(\rd)$ satisfy $\supp g\subset
[-1/8,1/8]^d$,  and $|\hat{g}|\geq 1$ on $[-2,2]^d$. The proof of each part of the
Lemma is based on the appropriate estimate for $V_{g}f$.

Let us first show that $f\in \M{p,q}{t,0}(\rd)$.
We have
\begin{align*}&\left|\intrd e^{-2\pi i (\o-k)y}\psi(y-k)g(y-x)dy\right|\\
\quad&=\left|\intrd \psi(y-k)g(y-x)\{(1+|\o-k|^2)^{-d}(I-\Delta_y)^d e^{-2\pi i
(\o-k)y}\} \,dy\right|\\
\quad&=\frac1{(1+|\o-k|^2)^d}\left|\sum_{|\b_1+\b_2|\leq 2d}C_{\b_1,\b_2}\intrd
\partial^{\b_1}(T_k\psi)(y)(\partial^{\b_2} g)(x-y)e^{-2\pi i (\o-k)y} dy\right|\\
\quad&\leq \frac
C{(1+|\o-k|^2)^d}\sum_{|\b_1+\b_2|\leq
2d}(|T_k(\partial^{\b_1}\psi)|\ast|\partial^{\b_2}g|)(x).
\end{align*}

Hence
 \begin{align}&\|f\|_{\M{p,q}{t,0}}\asymp \|V_g f\|_{L^{p,q}_{t,0}} \notag \\
 \quad&=\left(\intrd\left(\intrd\left|\sum_{k\not=0}|k|^{-\frac{d}q-\epsilon-t}\intrd
e^{-2\pi i (\o -k)y}\psi(y-k)g(y-x)\,dy\right|^p\la x\ra^{tp}
dx\right)^{\frac{q}p}\,d\o\right)^{\frac1q} \notag \\
\quad&\leq C
\Big(\intrd\big(\sum_{k\not=0}|k|^{-\frac{d}q-\epsilon-t}\frac
1{(1+|\o-k|^2)^d}\sum_{|\b_1+\b_2|\leq
2d}\| |T_k(\partial^{\b_1}\psi)|\ast
|\partial^{\b_2}g| \|_{L^p_t}\big)^q
d\o\Big)^{\frac1q}.
\label{1stkeyest}
\end{align}

 Using Young's inequality: $\| |T_k(\partial^{\b_1}\psi)|\ast|\partial^{\b_2}g|
\|_{L^p_t}\lesssim \|
T_k\partial^{\b_1}\psi\|_{L^1_t}\,\|\partial^{\b_2}g\|_{L^p_t}$,    and the
estimate $\|T_k\partial^{\b_1}\psi\|_{L^1_t}\leq \la
k\ra^t\|\partial^{\b_1}\psi\|_{L^1_t}$, we can control~\eqref{1stkeyest} by

 \begin{align}
 C  & \left(\intrd\left(\sum_{k\not=0}|k|^{-\frac{d}q-\epsilon}\frac
1{(1+|\o-k|^2)^d}\right)^q d\o\right)^{\frac1q} \notag \\
 \quad& \leq C
\left(\sum_{\ell\in\zd}\int_{\ell+[-1/2,1/2]^d}\left(\sum_{k\not=0}|k|^{-\frac{d}q-\epsilon}\frac
1{(1+|\o-k|^2)^d}\right)^q d\o\right)^{\frac1q} \notag \\
 \quad& \leq \tilde{C}
\left(\sum_{\ell\in\zd}\left(\sum_{k\not=0}|k|^{-\frac{d}q-\epsilon}\frac
1{(1+|\ell-k|^2)^d}\right)^q\right)^{\frac1q}\notag \\
 \quad &= \tilde{C} \Big\| |k|^{-\frac{d}q-\epsilon}\ast
\frac1{(1+|k|^2)^d}\Big\|_{\ell^q} <\infty,
\label{2nkeyest}
 \end{align}

 since $\{|k|^{-\frac{d}q-\epsilon}\}_{k\not=0}\in\ell^q.$\par

Next, we prove~\eqref{est1}.  Since $V_g f_\lambda\phas
=\lambda^{-d}V_{g_{\lambda^{-1}}} f(\lambda x,\lambda^{-1}\o) $, we obtain
\begin{equation*}
\|V_g
f_\lambda\|_{L^{p,q}_{t,0}}=\lambda^{-d(1+\frac1p-\frac1q)}\left(\intrd\left(\intrd|V_{g_{\lambda^{-1}}}
f\phas|^p\la \lambda^{-1} x\ra^{p t}\,dx\right)^{\frac{q}p}d\o\right)^{\frac1q}.
\end{equation*}
Observe that $$\la \lambda^{-1} (x+\ell)\ra\geq \lambda^{-1} \la \lambda^{-1} \ell
\ra \geq \la \lambda^{-1} \ell \ra$$ and $\supp
g((\cdot -x)/\lambda)\subset \ell+[-1/4,1/4]^d$, for all $0< \lambda\leq1$,
$x\in\ell+[-1/8,1/8]^d$. Since $\supp \psi(\cdot-k)\subset k+[-1/2,1/2]^d$ and
$\psi(t-k)=1$ if $t\in k+[-1/4,1/4]^d$, the inner integral can be estimated as
follows:
\begin{align*} &\left(\intrd |V_{g_{\lambda^{-1}}} f\phas|^p\la \lambda^{-1}
x\ra^{pt}\,dx\right)^{\frac1p}\\
 \quad &\geq\Big(\sum_{\ell\not=0}\int_{\ell+[-1/8,1/8]^d}\Big|\sum_{k\not=0}
|k|^{-d/q-\epsilon-t}\intrd e^{-2\pi i (\o
-k)y}\psi(y-k)\overline{g(\frac{y-x}{\lambda})}\,dy\Big|^p \la \lambda^{-1}
x\ra^{pt}\,d x\Big)^{\frac1p}\\
 \quad&\gtrsim\Big(\sum_{\ell\not=0}(|\ell|^{-\frac{d}q-\epsilon-t}\lambda^d
|\hat{g}(-\lambda(\o-\ell))|\,\lambda^{-t}|\ell|^t)^p\Big)^{\frac1p}\\
 \quad&\gtrsim\Big(\sum_{\ell\not=0}(|\ell|^{-\frac{d}q-\epsilon}\lambda^{d-t}
|\hat{g}(-\lambda(\o-\ell))|)^p\Big)^{\frac1p}.
 \end{align*}
Consequently,
\begin{align*}
\|V_g f_\lambda\|_{L^{p,q}_{t,0}}
&=\lambda^{-d(1+\frac1p-\frac1q)}\left(\intrd\left(\intrd|V_{g_{\lambda^{-1}}}
f\phas|^p\la \lambda^{-1} x\ra^{p t}\,dx\right)^{\frac qp}d\o\right)^{\frac1q}\\
& \gtrsim \lambda^{-d(1+\frac1p-\frac1q)}\,
\left(\intrd\Big(\sum_{\ell\not=0}(|\ell|^{-\frac dq-\epsilon}\lambda^{d-t}
|\hat{g}(-\lambda(\o-\ell))|)^p\Big)^{\frac qp}d\o\right)^{\frac1q}\\
& = \lambda^{d-t - d/q}\, \lambda^{-d(1+\frac1p-\frac1q)}\,
\left(\intrd\Big(\sum_{\ell\not=0}(|\ell|^{-\frac dq-\epsilon}
|\hat{g}(\o+\lambda \ell)|)^p\Big)^{\frac qp}d\o\right)^{\frac1q}\\
& \gtrsim \lambda^{-t -\frac dp}\, \left(\int_{|\o|\leq 1}\Big(\sum_{|\ell|\leq
\tfrac{1}{\lambda}}(|\ell|^{-\frac dq-\epsilon}
|\hat{g}(\o+\lambda \ell)|)^p\Big)^{\frac qp}d\o\right)^{\frac1q}\\
& \gtrsim \lambda^{-t -\frac dp}\, \left(\int_{|\o|\leq 1}\Big(\sum_{|\ell|\leq
\tfrac{1}{\lambda}}(|\ell|^{-\frac dq-\epsilon})^p\Big)^{\frac
qp}d\o\right)^{\frac1q}\\
& = \lambda^{-t -\frac dp}\, \Big(\sum_{|\ell|\leq
\tfrac{1}{\lambda}}(|\ell|^{-\frac dq-\epsilon})^p\Big)^{\frac1p}\gtrsim \lambda^{-t
-\frac dp} \, \lambda^{\frac dq + \epsilon} \Big(\sum_{|\ell|\leq
\tfrac{1}{\lambda}}\Big)^{\frac1p} \gtrsim \lambda^{-t -2\frac dp+ \frac dq +
\epsilon},
\end{align*}
which completes the proof.
\end{proof}

\begin{lemma}\label{i3negative}
Let $1\leq p, q \leq \infty$ be such that $(1/p,1/q)\in I_3$. Let
$\epsilon >0$, $t<0$, and $ 0< \lambda<1$.

\noindent a) If $t\leq -d$ define
\begin{equation}\label{i3neginfd}
f(x)=\lambda^{\tfrac{d}{q} -\tfrac{2d}{p} +2d}\sum_{k\neq
0}|k|^{-\tfrac{\epsilon}{2}}T_{\lambda^{2}k}\varphi(x),\quad \mbox{in}\,\, \cS'(\rd).
\end{equation}
Then there exists a constant $C>0$  such that
$\|f\|_{\M{p,q}{t,0}}\leq C $, uniformly with respect to $\lambda$. Moreover,
\begin{equation}\label{i3negest}\nm{f_{\lambda}}{\M{p,q}{t,0}}\gtrsim \lambda^{d\mu_{2}(p,q)+\epsilon},\quad\forall\,\,0<\lambda<1.
\end{equation}

\noindent b) If $-d < t< 0$, choose a positive integer $N$ large enough such that
$\tfrac{1}{N} <\tfrac{p-1}{2}-\tfrac{pt}{2d}$. Define
\begin{equation}\label{i3negd0}
f(x)=\lambda^{\tfrac{d}{q}}\sum_{k\neq 0}|k|^{d(\tfrac{2}{Np}-1)
-\tfrac{\epsilon}{N}}T_{\lambda^{N}k}\varphi(x),\quad\mbox{in}\,\,\cS'(\rd).
\end{equation}
Then the conclusions of part a) still hold. 
\end{lemma}

\begin{proof}
\noindent $a)$ For the range of $p,q$ being considered, $
\tfrac{d}{q}+2d -\tfrac{2d}{p}=d\mu_{2}(p,q)+2d\geq 0$, and so if $\lambda < 1$, then $\lambda^{\tfrac{d}{q}+2d -\tfrac{2d}{p}}< 1$.

Next, notice that $\mathcal{G}(\varphi, \lambda^{2}, 1)$ is a Gabor frame. So, to
check that $f \in \M{p,q}{t,0}$ we only need to verify that the sequence
$c=\{c_{k\ell}\}=\{|k|^{-\tfrac{\epsilon}{2}}\delta_{\ell,0}, k\neq 0\}_{k, \ell \in
\Z} \in \ell^{p,q}_{t,0}$. But, the condition $t\leq -d$ guarantees this, since
$$\nm{c}{\ell^{p,q}_{t,0}}=\lambda^{\tfrac{d}{q}+2d -\tfrac{2d}{p}}\bigg(\sum_{k\neq
0}|k|^{-p\epsilon/2}(1+|k|^{2})^{pt/2}\bigg)^{1/p} \leq C.$$

Next, as in the proof of Lemma~\ref{istar1}, we have
$$\|f_{\lambda}\|_{\M{p,q}{t,0}}=\sup_{\|g\|_{\M{p',q'}{-t,
0}}=1}|\ip{f_{\lambda}}{g}|\geq \|\varphi\|_{\M{p',
q'}{-t,0}}^{-2}|\ip{f_{\lambda}}{\varphi}|.$$
In this case,  $$\ip{f_{\lambda}}{\varphi}=\lambda^{2d+d\mu_{2}(p,q)}\sum_{k\neq
0}|k|^{-\epsilon/2}V_{\varphi}\varphi_{\lambda}(\lambda
k,0)=\lambda^{2d+d\mu_{2}(p,q)}\sum_{k\neq
0}|k|^{-\epsilon/2}(1+\lambda^2)^{-d/2}e^{\tfrac{-\pi
\lambda^{2}|k|^{2}}{\lambda^{2}+1}}.$$  Therefore, if $\lambda < 1$,
\begin{align*}
\|f_{\lambda}\|_{\M{p,q}{t,0}} & \geq C \lambda^{2d+d\mu_{2}(p,q)}\sum_{k\neq
0}|k|^{-\epsilon/2}(1+\lambda^2)^{-d/2}e^{\tfrac{-\pi
\lambda^{2}|k|^{2}}{\lambda^{2}+1}}\\
& \geq C  \lambda^{2d+d\mu_{2}(p,q)} \sum_{k\neq 0}|k|^{-\epsilon/2} \geq C
\lambda^{2d+d\mu_{2}(p,q)}  \sum_{0< |k| <
\tfrac{1}{\lambda^{2}}}|k|^{-\epsilon/2}\\
& \geq C \lambda^{2d+d\mu_{2}(p,q)}  \lambda^{ \epsilon}
\lambda^{-2d}=C\lambda^{d\mu_{2}(p,q)+\epsilon}
\end{align*}  which completes the proof of part $a)$.

\noindent  $b)$ If $p\geq 1$, the assumptions $-d<t< 0$ and
$\frac1 N < \tfrac{p-1}{2}-\tfrac{pt}{2d}$  are sufficient to
prove that $f \in \M{p,q}{t,0}.$ In addition, the main estimate is
that
\begin{align*}
\|f_{\lambda}\|_{\M{p,q}{t,0}} & \geq C \lambda^{d/q} \sum_{k\neq
0}|k|^{d(\tfrac{2}{N p}-1) -\tfrac{\epsilon}{N}}(1+\lambda^2)^{-d/2}e^{\tfrac{-\pi
\lambda^{2(N-1)}|k|^{2}}{\lambda^{2}+1}}\\
& \geq C  \lambda^{d/q}   \sum_{0< |k| < \tfrac{1}{\lambda^{N}}}|k|^{d(\tfrac{2}{N
p}-1) -\tfrac{\epsilon}{N} }\geq C\lambda^{d\mu_{2}(p,q)+\epsilon}.
\end{align*}
\end{proof}

We now state results similar to the above lemmas when the weight is in the frequency
variable.

\begin{lemma}\label{i1} For $s\leq 0$, $0<\lambda\leq1$, consider the family of
functions
\begin{equation}\label{ModGaus}
f(x)=\lambda^{s}M_{\lambda^{-1}e_1}\f(x),\ e_1=(1,0,0,\ldots,0).
\end{equation}
Then there exists a constant $C>0$  such that
$\|f\|_{\M{p,q}{0,s}}\leq C $, uniformly with respect to $\lambda$. Moreover,
\begin{equation}\label{bo2}
\|f_\lambda\|_{\M{p,q}{0,s}}\gtrsim \lambda^{s-\frac{d}p},\quad
\forall\,\,0<\lambda\leq1.
\end{equation}
\end{lemma}
\begin{proof}
 We have
\begin{align*}
\|f\|_{\M{p,q}{0,s}}&\asymp
\|V_\f f(x,\omega)\langle
\o\rangle^s\|_{L^{p,q}}=\lambda^{s}\|V_\f
\f(x,\omega-\lambda^{-1} e_1)\langle
\o\rangle^s\|_{L^{p,q}}\\
&=\lambda^{s}\|V_\f \f(x,\omega)\langle \o+\lambda^{-1}
 e_1\rangle^s\|_{L^{p,q}}\lesssim\lambda^{s}\lambda^{-s}\|V_\f
\f\la\o\ra^{-s}\|_{L^{p,q}}\lesssim1,
\end{align*}
 where we have used again the fact that  the weight $\la\cdot\ra^s$ is
$\la\cdot\ra^{-s}$-moderate. Thus
the
 functions $f$ have norms in
 $\M{p,q}{0,s}$ uniformly
 bounded with respect to
 $\lambda$. Let us now  estimate
 $\|f_\lambda\|_{\M{p,q}{0,s}}$
 from below. We have
 \[
 f_\lambda(x)=\lambda^{s}M_{e_1}\f_\lambda(x).
 \]
 By  using  \eqref{aggiunta.0},  we obtain
 \begin{align*}
 \|f_\lambda\|_{\M{p,q}{0,s}}&=\lambda^s\|V_\f \f_{\lambda}(x,\o-e_1)\la\o
\ra^s\|_{L^{p,q}}\\
  &\gtrsim \lambda^{s-\frac{d}p} \Bigg(\int
 e^{-\pi q|\o|^2} \langle
 \o+e_1\rangle^{qs} d\o\Bigg)^{\frac1q} \gtrsim \lambda^{s-\frac{d}p},
 \end{align*}
 as desired.
 \end{proof}

\begin{lemma}\label{fistar2} Let $1\leq p, q \leq \infty$ be such that $(1/p,1/q)\in
I^*_2$. Assume that $s\leq 0$, $\epsilon >0$ and $\lambda >1 $.

\noindent a) If  $q\geq 2$ and $s\leq 0$, or $1\leq q \leq 2$ and $s\leq -d$,   define
$$
f(x)=\sum_{\ell\neq 0}|\ell|^{d(\frac1q-1) -\epsilon}e^{2\pi i \lambda^{-1}\ell\cdot
x}\varphi(x)=\sum_{\ell\neq
0}|\ell|^{d(\frac1q-1)-\epsilon}M_{\lambda^{-1}\ell}\varphi(x),\quad  \mbox{in}\,\, \cS'(\rd).$$
Then there exists a constant $C>0$  such that
$\|f\|_{\M{p,q}{0,s}}\leq C $, uniformly with respect to $\lambda$. Moreover,
\begin{equation}\label{Gaugffreqneg}
\|f_{\lambda}\|_{\M{p,q}{0,s}}\gtrsim \lambda^{d(\frac1q -1)-\epsilon},\,\quad
\forall \lambda > 1.
\end{equation}

\noindent b) If  $1\leq q \leq 2$ and $-d< s <0$, choose a positive integer $N$
such that $\tfrac{1}{N} < -\tfrac{sq}{d}$, and define

$$
f(x)=\sum_{\ell\neq 0}|\ell|^{d(\tfrac{1}{Nq}-1) -\epsilon/N}e^{2\pi i
\lambda^{-N}\ell\cdot x}\varphi(x)=\sum_{\ell\neq
0}|\ell|^{d(\tfrac{1}{Nq}-1)-\epsilon/N}M_{\lambda^{-N}\ell}\varphi(x), \quad \mbox{in}\,\,
\cS'(\rd).$$  Then the conclusions of part a) still hold. 
\end{lemma}

\begin{proof}
\noindent $a)$ First of all notice that $\mathcal{G}(\varphi, 1, \lambda^{-1})$ is a
frame. In addition,  $q\geq 2$ is equivalent to $1/q -1 \leq -1/q$. Thus, for all $s\leq 0$, $\{|\ell|^{d(1/q -1) - \epsilon}, \ell\neq 0\} \in
\ell^{q}_{s},$ which ensures that the function $f$  defined above belongs to
$\M{p,q}{0,s}$.  This is also true when $1\leq q \leq 2$ and $s\leq -d$.

To prove~\eqref{Gaugffreqneg} we follow the proof of Lemma~\ref{istar1}. In
particular, we have
$$
\|f_{\lambda}\|_{\M{p,q}{0,s}}  \geq C  \sum_{\ell\neq 0}|\ell|^{d(\frac1q
-1)-\epsilon}(1+\lambda^2)^{-d/2}e^{\tfrac{-\pi |\ell|^{2}}{\lambda^{2}+1}} \geq C
\lambda^{-d} \sum_{0< |\ell|\leq {\lambda}}|\ell|^{d(1/q -1)-\epsilon}
e^{\tfrac{-\pi |\ell|^{2}}{\lambda^{2}+1}},$$
from which \eqref{Gaugffreqneg} follows.

\noindent $ b)$ In this case, $\mathcal{G}(\varphi, 1, \lambda^{-N})$ is a frame.
Moreover, the choice of $N$ insures that $d(1/(Nq) -1) +s <-d$ which is enough to
prove that $f\in \M{p,q}{0,s}$, and that $\|f\|_{\M{p,q}{0,s}}\leq C .$ 
Relation \eqref{Gaugffreqneg} now follows from
\begin{align*}
\|f_{\lambda}\|_{\M{p,q}{0,s}} & \geq C  \sum_{\ell\neq 0}|\ell|^{d(\tfrac{1}{Nq}
-1)-\epsilon/N}(1+\lambda^2)^{-d/2}e^{\tfrac{-\pi
\lambda^{-2N+2}|\ell|^{2}}{\lambda^{2}+1}}\\
& \geq C \lambda^{-d} \sum_{0< |\ell|\leq {\lambda^{N}}}|\ell|^{d(\tfrac{1}{Nq}
-1)-\epsilon/N} e^{\tfrac{-\pi \lambda^{-2N+2}|\ell|^{2}}{\lambda^{2}+1}}\geq C
\lambda^{d(\frac1q -1)-\epsilon}.
\end{align*}
\end{proof}

The next lemma is proved similarly to Lemma~\ref{lemm1}, so we omit its proof.

\begin{lemma}\label{lemm1b} Let $1\leq p,q \leq\infty$,  $s\leq 0$,
$\epsilon>0$. Suppose that $\psi\in \cS(\rd)$ satisfies $\supp\psi\subset [-1/2,1/2]^d$
and $\psi=1$ on $[-1/4,1/4]^d$.

\noindent a) If $1\leq q < \infty$, define
$
f(y)=\sum_{k\in\zd\setminus\{0\}} | k|^{-\frac{d}q-\epsilon-s} M_{k}T_k
\psi(y),\quad\mbox{in}\,\,\cS'(\rd).
$
Then, $f\in \M{p,q}{0,s}(\rd)$ and
\begin{equation}\label{est1b}
\|f_\lambda\|_{\M{p,q}{0,s}}\gtrsim
\lambda^{-d(\frac2p-\frac1q)+\epsilon+s},\,\quad\forall\,\,0<\lambda\leq 1.
\end{equation}

\noindent b) If $q=\infty$, let

\begin{equation}\label{es22}
f(y)=\sum_{k\not=0} |k|^{-s}e^{2\pi i k y}T_k\psi(y),\quad\mbox{in}\,\,\cS'(\rd).
\end{equation}
Then $f\in \M{p,\infty}{0,s}$ and
\begin{equation}\label{est2b}\|f_\lambda\|_{\mathcal{M}^{p,\infty}_{0,s}}\gtrsim
\lambda^{-
\frac{2d}p+s},\quad\forall\,\,0<\lambda\leq1.
\end{equation}
\end{lemma}

\begin{lemma}\label{i3postivefreq}
Let $1\leq p, q \leq \infty$ be such that $(1/p,1/q)\in I_3$. Let
$\epsilon >0$, $s\geq 0$  and $ 0< \lambda< 1$. Assume that $p>1$,
and choose a positive integer $N$ such that $\frac1N <
\tfrac{p-1}{2}$. Define
\begin{equation}\label{i3posfreq}
f(x)=\lambda^{\tfrac{d}{q}}\sum_{k \neq 0}|k|^{d(\tfrac{2}{N p}-1)
-\tfrac{\epsilon}{N}}T_{\lambda^{N}k}\varphi(x), \quad
\mbox{in}\,\,\cS'(\rd).
\end{equation}
Then, there exists a constant $C>0$ such that
$\|f\|_{\M{p,q}{0,s}}\leq C $, uniformly with respect to
$\lambda$. Moreover,
$$\nm{f_{\lambda}}{\M{p,q}{0,s}}\gtrsim \lambda^{d\mu_{2}(p,q)+\epsilon}.$$
\end{lemma}

\begin{proof} In this case, $\mathcal{G}(\varphi,  \lambda^{N}, 1)$ is a frame.
 The condition $\frac1N < \tfrac{p-1}{2}$ is equivalent to $\tfrac{2}{Np}-1 < -\tfrac{1}{p}$ which is enough to show that $\{|k|^{d(\tfrac{2}{Np}-1) - \tfrac{\epsilon}{N}}\}_{k\neq 0} \in \ell^{p}$. Therefore, $ f \in \M{p,q}{0,s}$ with $\nm{f}{\M{p,q}{0,s}}\leq C$ where $C$ is a
universal constant. The rest of the proof is an adaptation of the proof of Lemma~\ref{i3negative}.
\end{proof}

Notice that, the previous lemma excludes the case $p=1$. We prove
this last case by considering the dual case. Observe that the case
$(1/\infty,1/\infty)\in I_{1}^{*}\cap I_{3}^{*}$ was already
considered in dealing with the region $I_{1}^{*}$.

\begin{lemma}\label{inftynegfreq}
Let $1\leq q \leq \infty $ be such that $(1/\infty,1/q)\in
I_3^{*}$. Let $\epsilon >0$, $s\leq 0$  and $  \lambda> 1$.

\noindent a) If $1< q <2$, choose a positive integer $N$ such that
$\frac3N < q-1$. Define
\begin{equation}\label{inftyneqfreqa}
f(x)=\lambda^{d(1-\tfrac{2}{q})}\sum_{\ell \neq
0}|\ell|^{d(\tfrac{3}{N q}-1)
-\tfrac{\epsilon}{N}}M_{\lambda^{-N}\ell}\varphi(x),\quad
\mbox{in}\,\,\cS'(\rd).
\end{equation}
Then there exists a constant $C>0$  such that
$\|f\|_{\M{p,q}{0,s}}\leq C $, uniformly with respect to
$\lambda$. Moreover,
$$\nm{f_{\lambda}}{\M{\infty,q}{0,s}}\gtrsim
\lambda^{\tfrac{d}{q}-\epsilon}.$$

\noindent b) If $2\leq  q <\infty$, choose a positive integer $N$
such that $N>2+q$. Define
\begin{equation}\label{inftyneqfreqb}
f(x)=\lambda^{d+\tfrac{d(2-N)}{q}}\sum_{\ell \neq
0}|\ell|^{d(\tfrac{N-1}{N q}-1)
-\tfrac{\epsilon}{N}}M_{\lambda^{-N}\ell}\varphi(x),\quad
\mbox{in}\,\,\cS'(\rd).
\end{equation}
Then the conclusions of part a) still hold. 

\noindent c) If $ q=1$ and $s\leq -d$, define
\begin{equation}\label{infty1neqfreqb}
f(x)=\sum_{\ell \neq
0}|\ell|^{-\tfrac{\epsilon}{2}}M_{\lambda^{-2}\ell}\varphi(x),
\quad\mbox{in}\,\,\cS'(\rd).
\end{equation}
Then there exists a constant $C>0$  such that
$\|f\|_{\M{\infty,1}{0,s}}\leq C $, uniformly with respect to
$\lambda$. Moreover,
$$\nm{f_{\lambda}}{\M{\infty,1}{0,s}}\gtrsim \lambda^{d-\epsilon}.$$

\noindent d) If $ q=1$ and $-d < s < 0 $, choose a positive
integer $N$ such that $\tfrac{1}{N} < \tfrac{-s}{2d}$. Define
\begin{equation}\label{infty1neqfreqb2}
f(x)=\sum_{\ell \neq
0}|\ell|^{d(\tfrac{2}{N}-1)-\tfrac{\epsilon}{N}}M_{\lambda^{-N}\ell}\varphi(x),
\quad\mbox{in} \,\,\cS'(\rd).
\end{equation}
Then the conclusions of part c) still hold. 
\end{lemma}

\begin{proof}
\noindent $a)$ In this case, $\mathcal{G}(\varphi,  1,
\lambda^{-N})$ is a frame. The hypotheses  $1< q <2$ and $\lambda
>1$ imply that $\lambda^{d(1-\tfrac{2}{q})}<1$. In addition, the
condition $\frac3N < q-1$ is equivalent to $\tfrac{3}{Nq}-1 <
-\tfrac{1}{q}$ which is enough to show that
$\{|\ell|^{d(\tfrac{3}{Nq}-1) - \tfrac{\epsilon}{N}}\}_{\ell \neq
0} \in \ell^{q}_{s}$. Therefore, $ f \in \M{\infty,q}{0,s}$ with
$\nm{f}{\M{\infty,q}{0,s}}\leq C$ where $C$ is a universal
constant. The rest of the proof is an adaptation of the proof of
Lemma~\ref{i3negative}.

\noindent $b)$  Assume that $2\leq q < \infty$. The proof is
similar to the above with the following differences:  $N > q+2$
and $\lambda >1$ imply that $\lambda^{d(1+\tfrac{2-N}{q})}<1$. In
addition, the condition $q\geq 2$ implies that $\tfrac{N-1}{Nq}-1
< -\tfrac{1}{q}$. This is enough to show that
$\{|\ell|^{d(\tfrac{N-1}{Nq}-1) - \tfrac{\epsilon}{N}}\}_{\ell
\neq 0} \in \ell^{q}_{s}$. Therefore, $ f \in \M{\infty,q}{0,s}$
with $\nm{f}{\M{\infty,q}{0,s}}\leq C$ where $C$ is a universal
constant.

\noindent $c)$ In this case, $\mathcal{G}(\varphi,  1,
\lambda^{-2})$ is a frame. The fact that $s\leq -d$ implies that
$\{|\ell|^{- \tfrac{\epsilon}{2}}\}_{\ell \neq 0} \in
\ell^{1}_{s}$. Therefore, $ f \in \M{\infty,1}{0,s}$ with
$\nm{f}{\M{\infty,1}{0,s}}\leq C$ where $C$ is a universal
constant. The rest of the proof is an adaptation of the proof of
Lemma~\ref{i3negative}.

\noindent $d)$ In this case, $\mathcal{G}(\varphi,  1,
\lambda^{-N})$ is a frame. The fact that $-d< s<0$ and the choice
of $N$ imply that $d(\tfrac{2}{N}-1) + s < -d$. Therefore
$\{|\ell|^{d(\tfrac{2}{N}-1)- \tfrac{\epsilon}{2}}\}_{\ell \neq 0}
\in \ell^{1}_{s}$. Therefore, $ f \in \M{\infty,1}{0,s}$ with
$\nm{f}{\M{\infty,1}{0,s}}\leq C$ where $C$ is a universal
constant. The rest of the proof is an adaptation of the proof of
Lemma~\ref{i3negative}.

\end{proof}

We finish this subsection by proving lower bound estimates for the dilation of
functions that are compactly supported either in the time or in the frequency
variables.

\begin{lemma}\label{altertf}
Let $u\in \cS(\rd)$, $\lambda\in(0,\infty)$ and $1\leq p, q \leq \infty$.

(i) If $u$ is supported in a compact set $K\subset \rd$, then, for every $t\in \R$,
and $\lambda \geq 1$
\begin{equation}\label{spK}
\| u_{\lambda}\|_{\M{p,q}{t,0}}\gtrsim \lambda^{-d(1-\frac1 q)}\min \{ 1,
\lambda^{-t}\}.
\end{equation}

(ii) If $\hat{u}$ is supported  in a compact set $K\subset \rd$, then, for every
$s\in\R$,  and $\lambda \leq 1$
\begin{equation}\label{spKf}
\| u_{\lambda}\|_{\M{p,q}{0,s}}\gtrsim C \lambda^{-\frac d p } \min \{ 1, \lambda^s\}.
\end{equation}
\end{lemma}
\begin{proof}
We use the dilation properties for the  Sobolev spaces (Bessel potential spaces)
$H_s^p(\rd)$ (see, e.g., \cite[Proposition 3]{rusic}):
$$C^{-1 } \lambda^{-\frac d p}  \min \{ 1, \lambda^s\}  \|u\|_{H_s^p} \leq
\|u_{\lambda}\|_{H_s^p}\leq C \lambda^{-\frac d p}\max \{ 1,
\lambda^s\}\|u\|_{H_s^p}, \quad 1\leq p\leq \infty, \,\,s>0.$$
 $(i)$ Let $u$   be supported
in a compact set $K\subset
\rd$, we have $u\in
\mathcal{M}^{p,q}\Leftrightarrow u\in
\Fur L^q$, and
\begin{equation}
C_K^{-1} \|u\|_{\mathcal{M}^{p,q}}\leq
\|u\|_{\cF L^q}\leq C_K
\|u\|_{\mathcal{M}^{p,q}},
\end{equation}
where $C_K>0$ depends only on
 $K$ (see, e.g., \cite{fe89-1, ko09}).  Hence, if $\lambda\geq 1$,
 \begin{align*} \|u_\lambda\|_{\M{p,q}{t,0}} &\asymp \|\la \cdot\ra^t
u_\lambda\|_{\M{p,q}{}}\asymp \|\la \cdot\ra^t u_\lambda\|_{\cF L^q} \asymp
\|\cF^{-1} (u_\lambda)\|_{H^q_t}
 =\lambda^{-d}\| (\cF^{-1} u)_{\lambda^{-1}}\|_{H^q_t}\\
 &\geq \lambda^{-d}(\lambda^{-1})^{-\frac d q}\min \{1, \lambda^{-t}\}.
 \end{align*}

 $(ii)$ Now let $\hat{u}$ be supported in a compact set $K\subset \rd$. We have $u\in
\mathcal{M}^{p,q}\Leftrightarrow u\in
L^p$, and
\begin{equation*}
C_K^{-1} \|u\|_{\mathcal{M}^{p,q}}\leq\|u\|_{ L^p}\leq C_K\|u\|_{\mathcal{M}^{p,q}},
\end{equation*}
where $C_K>0$ depends only on  $K$ (again, see, e.g., \cite{fe89-1}). Arguing as in
part (i) above  with $0<\lambda\leq 1$,
 \begin{equation*} \|u_\lambda\|_{\M{p,q}{0,s}} \asymp \|\la D\ra^s
u_\lambda\|_{\mathcal{M}^{p,q}}\asymp \|\la D\ra^s u_\lambda\|_{L^p} \asymp
\|u_\lambda\|_{H_s^p}
 \geq C\lambda^{-\frac d p } \min \{1, \lambda^s\} \| u\|_{H_s^p}
 \end{equation*}
and the proof is completed. 
\end{proof}

\subsection{Sharpness of Theorems \ref{xdil} and \ref{mainfreq}.}\,

We are now in position to state and prove the sharpness of the results obtained in
Section~\ref{main}. In particular, Theorem~\ref{xdil} is optimal in the following
sense:

\begin{theorem}\label{sharp31} Let $1\leq p, q \leq \infty.$

\noindent (A) If $t \geq 0$ then the following statements hold:

\noindent Assume that there exist  constants $C>0$, and  $\alpha, \beta \in \R$ such
that
\begin{equation}\label{sharp1}
C^{-1}\, \lambda^{\beta}\|f\|_{\M{p,q}{t,0}}\leq \|U_{\lambda}f\|_{\M{p,q}{t,0}}\leq
C \lambda^{\alpha}\|f\|_{\M{p,q}{t,0}}
\quad \forall f \in \M{p,q}{t,0}\quad {\textrm and} \quad \lambda \geq 1,
\end{equation} then, $\alpha \geq d\mu_{1}(p,q)$, and $\beta \leq d\mu_{2}(p,q)-t.$

\noindent  Assume that there exist  constants $C>0$, and  $\alpha, \beta \in \R$
such that
\begin{equation}\label{sharp2}
C^{-1}\, \lambda^{\alpha}\|f\|_{\M{p,q}{t,0}}\leq
\|U_{\lambda}f\|_{\M{p,q}{t,0}}\leq C \lambda^{\beta}\|f\|_{\M{p,q}{t,0}}\quad
\forall f \in \M{p,q}{t,0}\quad {\textrm and} \quad 0<\lambda \leq 1,
\end{equation} then, $\alpha \geq d\mu_{1}(p,q)$, and $\beta \leq d\mu_{2}(p,q)-t.$

\noindent (B)
If $t\leq 0$ then the following statements hold:

\noindent Assume that there exist  constants $C>0$, and  $\alpha, \beta \in \R$ such
that
\begin{equation}\label{sharp3}
C^{-1}\, \lambda^{\beta}\|f\|_{\M{p,q}{t,0}}\leq \|U_{\lambda}f\|_{\M{p,q}{t,0}}\leq
C \lambda^{\alpha}\|f\|_{\M{p,q}{t,0}}\quad \forall f \in \M{p,q}{t,0}\quad {\textrm
and} \quad \lambda \geq 1,
\end{equation}
 then, $\alpha \geq d\mu_{1}(p,q)-t$, and $\beta \leq d\mu_{2}(p,q).$

\noindent  Assume that there exist  constants $C>0$, and  $\alpha, \beta \in \R$
such that
\begin{equation}\label{sharp4}
C^{-1}\, \lambda^{\alpha}\|f\|_{\M{p,q}{t,0}}\leq
\|U_{\lambda}f\|_{\M{p,q}{t,0}}\leq C \lambda^{\beta}\|f\|_{\M{p,q}{t,0}}\quad
\forall f \in \M{p,q}{t,0}\quad {\textrm and} \quad 0<\lambda \leq 1,
\end{equation} then, $\alpha \geq d\mu_{1}(p,q)-t$, and $\beta \leq d\mu_{2}(p,q).$

\end{theorem}
\begin{proof}

It will be enough to prove the upper half of each of the estimates, as the lower
halves will follow from the fact that $f=U_{\lambda}U_{1/\lambda}f$. Moreover, the
proof relies on analyzing the examples provided by the previous lemmas, and by
considering several cases.

\textbf{Case $1$: $(1/p,1/q)\in I^{*}_2$, $t\geq 0$.}
In this case we have  $ \lambda\geq 1$ and $\mu_1(p,q)=1/q-1$. Substitute
$f(x)=\varphi(x)=e^{-\pi|x|^{2}}$ in the upper half estimates~\eqref{sharp1} and use
Lemma~\ref{Gaussian} to obtain $$\lambda^{-d(1-1/q)}\lesssim
\|\varphi_\lambda\|_{\M{p,q}{t,0}}\leq C \lambda^\alpha
\|\varphi\|_{\M{p,q}{t,0}},$$ for all
$\lambda\geq 1$. This immediately implies that  $\alpha \geq -d(1-1/q)=d\mu_{1}(p,q)$.

\textbf{Case $2$: $(1/p, 1/q) \in I_{2}$, $t\leq0$.} This is the dual case to the
previous case and can be handled as follows. In this case we have $ \lambda\leq 1$
and $\mu_2(p,q)=1/q-1$. Assume that the upper-half estimate in~\eqref{sharp4} holds.
Notice that $(1/p, 1/q) \in I_2$ if and only if $(1/p', 1/q') \in I^{*}_{2}$, and
that $\lambda \leq 1$ if and only if $1/\lambda \geq 1$.
\begin{align*}\|f_{1/\lambda}\|_{\M{p', q'}{-t, 0}} & =
\sup|\ip{f_{1/\lambda}}{g}|=\lambda^{d}\sup|\ip{f}{g_{\lambda}}|\leq
\lambda^{d}\|f\|_{\M{p', q'}{-t, 0}} \sup\|g_{\lambda}\|_{\M{p,q}{t,0}}\\&\leq
\lambda^{d +\beta} \|f\|_{\M{p', q'}{-t, 0}} \sup\|g\|_{\M{p,q}{t,0}},
\end{align*}
where the supremum is taken over  all $g \in \cS$ and $\|g\|_{\M{p,q}{t,0}}=1$;
hence,
$$ \|f_{1/\lambda}\|_{\M{p', q'}{-t, 0}}\leq \lambda^{d +\beta} \|f\|_{\M{p',
q'}{-t, 0}}.
$$
Thus from Case $1$ above, $-\beta -d \geq d\mu_{1}(p', q')=d/q'-d$. Hence, $\beta
\leq d\mu_2(p,q)$.

\textbf{Case $3$: $(1/p,1/q)\in I_3$,  $t\geq 0$.}
In this case we have  $ \lambda\leq 1$ and $\mu_2(p,q)=-2/p+1/q$.
First assume that $1\leq q < \infty$ and that the upper-half estimate
in~\eqref{sharp3} holds for all $f \in \M{p,q}{t,0}$ and $0< \lambda <1,$ but that
$\beta >d\mu_{2}(p,q)-t.$ Then there is $\epsilon >0$ such that $\beta >
d\mu_{2}(p,q)-t+\epsilon.$ For this choice of $\eps>0$, we construct a
function $f$ as in~\eqref{lem1ex1} of Lemma~\ref{lemm1} such that:
$$\lambda^{d\mu_{2}(p,q) -t +\epsilon} \lesssim \|f_{\lambda}\|_{\M{p,q}{t,0}} \leq
C \lambda^{\beta}\|f\|_{\M{p,q}{t,0}}$$ for some $C>0$ and all $0<\lambda \leq 1$.
This leads to a contradiction on the choice of $\eps$. 

When $q=\infty$ the function given by~\eqref{es21} of Lemma~\ref{lemm1} gives the
optimal bound.

\textbf{Case $4$: $(1/p,1/q)\in I^{*}_3$,  $t\leq0$.} In this case,  $\lambda \geq
1$, and $\mu_1(p,q)=-2/p+1/q$. This is the dual of Case $3$, and a duality argument similar to the used
in Case $2$ above gives the result.

\textbf{Case $5$: $(1/p,1/q)\in I_1^{*}$,  $t\leq 0$.}
In this case, $\lambda \geq 1$, and $\mu_1(p,q)=-1/p$. Assume that the upper-half
estimate in~\eqref{sharp3} holds and that $\alpha < d\mu_1(p,q) -t$. Then, choose
$\epsilon >0$ and construct a function $f$ as in part $b)$ of Lemma~\ref{istar1}. A
contradiction immediately follows.

\textbf{Case $6$: $(1/p,1/q)\in I_1$,  $t\geq 0$.} In this case  $\lambda \leq 1$,
and $\mu_2(p,q)=-1/p$.
This is the dual of Case $5$.

\textbf{Case $7$: $(1/p,1/q)\in I_1^{*}$,  $t\geq 0$.} In this case  $\lambda \geq
1$, and $\mu_1(p,q)=-1/p$.
Assume that the upper-half estimate in~\eqref{sharp1} holds for all $f \in
\M{p,q}{t,0}$ and $ \lambda >1,$ but that $\alpha <d\mu_{1}(p,q).$ Then there is
$\epsilon >0$ such that $\alpha< d\mu_{1}(p,q)-\epsilon.$ For this choice of
$\eps>0$, we can now construct a function $f$ as in Lemma~\ref{istar1}, part $a)$,
such that:
$$\lambda^{d\mu_{1}(p,q) -\epsilon} \lesssim \|f_{\lambda}\|_{\M{p,q}{t,0}} \leq C
\lambda^{\alpha}\|f\|_{\M{p,q}{t,0}}$$ for some $C>0$ and all $\lambda \geq 1$. This
leads to a contradiction on the choice of $\eps$.

\textbf{Case $8$: $(1/p,1/q)\in I_1$, $t\leq 0$.} In this case $\lambda \leq 1$, and
$\mu_2(p,q)=-1/p$.
This is the dual of Case $7$.

\textbf{Case $9$: $(1/p,1/q)\in I_2^{*}$, $t\leq 0$.} In this case $\lambda \geq 1$,
and $\mu_1(p,q)=1/q-1$. The function constructed in Lemma~\ref{i2star} leads to the
result.

\textbf{Case $10$: $(1/p,1/q)\in I_2$,  $t\geq 0$.} In this case $\lambda \leq 1$,
and $\mu_2(p,q)=1/q-1$.
This is the dual of Case $9$.

\textbf{Case $11$: $(1/p,1/q)\in I_3$,  $t\leq 0$.}
In this case $\lambda \leq 1$, and $\mu_1(p,q)=-2/p+1/q$ and Lemma~\ref{i3negative}
can be used to conclude.

\textbf{Case $12$: $(1/p,1/q)\in I_3^{*}$, $t\geq 0$.}
In this case $\lambda \geq 1$, and $\mu_2(p,q)=-2/p+1/q$.
This is the dual of Case $11$. 
\end{proof}

We next consider the sharpness Theorem~\ref{mainfreq}.

\begin{theorem}\label{sharp32} Let $1\leq p, q \leq \infty.$

\noindent (A) If $s \geq 0$ then the following statements hold:

\noindent
Assume that there exist  constants $C>0$, and  $\alpha, \beta \in \R$ such that
\begin{equation}\label{sharp5}
C^{-1}\, \lambda^{\beta}\|f\|_{\M{p,q}{0,s}}\leq \|U_{\lambda}f\|_{\M{p,q}{0,
s}}\leq C \lambda^{\alpha}\|f\|_{\M{p,q}{0, s}}\quad \forall  f \in \M{p,q}{0,
s}\quad {\textrm and}\quad \lambda \geq 1,
\end{equation} then, $\alpha \geq d\mu_{1}(p,q) +s$, and $\beta \leq d\mu_{2}(p,q).$

\noindent
Assume that there exist  constants $C>0$, and  $\alpha, \beta \in \R$ such that
\begin{equation}\label{sharp6}
C^{-1}\, \lambda^{\alpha}\|f\|_{\M{p,q}{0, s}}\leq \|U_{\lambda}f\|_{\M{p,q}{0,
s}}\leq C \lambda^{\beta}\|f\|_{\M{p,q}{0, s}}\quad \forall f \in \M{p,q}{0, s}\quad
{\textrm and}\quad 0<\lambda \leq 1,
\end{equation} then, $\alpha \geq d\mu_{1}(p,q)+s$, and $\beta \leq d\mu_{2}(p,q).$

\noindent (B) If $s \leq 0$ then the following statements hold:

\noindent
Assume that there exist  constants $C>0$, and  $\alpha, \beta \in \R$ such that
\begin{equation}\label{sharp7}
C^{-1}\, \lambda^{\beta}\|f\|_{\M{p,q}{0,s}}\leq \|U_{\lambda}f\|_{\M{p,q}{0,
s}}\leq C \lambda^{\alpha}\|f\|_{\M{p,q}{0, s}}\quad \forall  f \in \M{p,q}{0,
s}\quad {\textrm and}\quad \lambda \geq 1,
\end{equation} then, $\alpha \geq d\mu_{1}(p,q) $, and $\beta \leq d\mu_{2}(p,q)+s.$

\noindent
Assume that there exist  constants $C>0$,  and  $\alpha, \beta \in \R$ such that
\begin{equation}\label{sharp8}
C^{-1}\, \lambda^{\alpha}\|f\|_{\M{p,q}{0, s}}\leq \|U_{\lambda}f\|_{\M{p,q}{0,
s}}\leq C \lambda^{\beta}\|f\|_{\M{p,q}{0, s}}\quad \forall f \in \M{p,q}{0, s}\quad
{\textrm and}\quad 0<\lambda \leq 1,
\end{equation} then, $\alpha \geq d\mu_{1}(p,q)$, and $\beta \leq d\mu_{2}(p,q)+s.$

\end{theorem}

\begin{proof}
As for the time weights, it is enough to prove the upper half of each estimates.
Moreover, in what follows we consider only
 $6$ of the $12$ cases to be proved, since the others are obtained by the same
duality argument used in the previous theorem.

\textbf{Case $1$: $(1/p,1/q)\in I_1$, $s\geq 0$.} In this case, $0<\lambda\leq1$ and
$\mu_2(p,q)=-1/p$. Assume there exist constants $C>0$ and $\beta\in \R$ such that
the upper-half estimate \eqref{sharp6} holds.
Taking the Gaussian $f=\f$ as in Lemma \ref{Gaussian}  and using \eqref{zerof}, we have
$$ \lambda^{-\frac d p} \lesssim  \|\f_\lambda\|_{\M{p,q}{0, s}}\lesssim
\lambda^\beta  \| \f\|_{\M{p,q}{0, s}},
$$
for all $0<\lambda\leq 1$. This gives $\beta\leq -d/p$.

\textbf{Case $2$: $(1/p,1/q)\in I_1$, $s\leq 0$.}   Here $\lambda \leq 1$ and we
test the upper-half estimate \eqref{sharp8}
on the  family of functions \eqref{ModGaus}. Using \eqref{bo2}, we obtain
$\b\leq s-d/p$.

\textbf{Case $3$: $(1/p,1/q)\in
I_2^*$, $s\geq0$.} Here $\lambda\geq1$, $\mu_1(p,q)=1/q-1$.
We assume the upper-half estimate \eqref{sharp5} and test it  on the dilated
Gaussian function in \eqref{infinityf},  obtaining $\a\geq d(1/q-1)+s$.

\textbf{Case $4$:  $(1/p,1/q)\in I_2^*$, $s\leq0$.} Here $\lambda\geq1$,
$\mu_1(p,q)=1/q-1$. We use a contradiction argument based on Lemma~\ref{fistar2}.

\textbf{Case $5$:  $(1/p,1/q)\in I_3$, $s\geq0$.}  Here
$\lambda\leq1$, $\mu_2(p,q)=-2/p+1/q$. The sharpness is obtained by
testing the upper-half estimate \eqref{sharp6} on the family of
functions $f_\lambda$, defined in  Lemma \ref{i3postivefreq} when
$p>1$.

If $p=1$ we consider the dual case, that is $(1/\infty,1/q)\in
I_{3}^{*}$, $s\leq0$.  Here $\lambda\geq1$, $\mu_1(\infty,q)=1/q$.
We use a contradiction argument based on Lemma~\ref{inftynegfreq}.

 \textbf{Case $6$:  $(1/p,1/q)\in I_3$, $s\leq0$.} Here $\lambda\leq1$,
$\mu_2(p,q)=-2/p+1/q$. The sharpness is obtained by testing the upper-half estimate
\eqref{sharp8} on the family of functions $f_\lambda$, defined in  Lemma
\ref{lemm1b}.

\end{proof}

\section{Applications}\label{applic}

\subsection{Applications to dispersive equations}\label{dispde}
\subsubsection{Wave equation.} Let us first recall the  Cauchy problem for
 the  wave equation:
\begin{equation}\label{cpw}
\begin{cases}
\partial^2_t u-\Delta_x u=0\\
u(0,x)=u_0(x),\,\,
\partial_t u (0,x)=u_1(x),\,\,
\end{cases}
\end{equation}
with $t\geq 0$, $x\in\R^d$, $d\geq1$, $\Delta_x=\partial^2_{x_1}+\dots
\partial^2_{x_d}$.  The formal solution  $u(t,x)$ is  given by
\begin{align*}
u(t,x)& =\intrd e^{2\pi i x\xi} \cos(2\pi t |\xi|) \widehat{u_0}(\xi)\,d\xi+\intrd
e^{2\pi i x\xi} \frac{\sin (2\pi  t |\xi|)}{2\pi |\xi|} \widehat{u_1}(\xi) \, d\xi,\\
&= H_{\sigma_{0}}u_{0}(x)+H_{\sigma_{1}}(x)
\end{align*}
with,
$\sigma_0(\xi)= \cos(2\pi t |\xi|)$ and $\sigma_1(\xi)=\frac{\sin (2\pi  t
|\xi|)}{2\pi|\xi|}$.

We recall that $H_{\sigma_{i}}$ $i=0, 1$, are examples of Fourier multipliers which
are defined by
\begin{equation}\label{FM}
H_{\sigma} f(x)=\intrd e^{2\pi i
x\xi}\sigma(\xi) \hat{f}(\xi)\,d\xi
\end{equation}
 where $\sigma$ is called the symbol.

The  boundedness of $H_{\sigma_{i}}$, $i=0, 1$ on modulation spaces was proved in
\cite{bgor, bo09} and in \cite{CNwave}. Moreover, some related local-in-time well
posedness results for certain nonlinear PDEs were also obtained in \cite{bo09,
CNwave} for initial data in modulation spaces.

\begin{proposition} \label{L1} Let $s\in \R$, and $1\leq p, q \leq \infty$. Then,
the solution $u(t, x)$ of~\eqref{cpw}  with initial data $(u_{0}, u_{1}) \in
\M{p,q}{0,s}\times \M{p,q}{0, s-1}$
satisfies
\begin{equation}\label{es1A}\|u(t,\cdot)\|_{\M{p,q}{0, s}}\leq C_0 (1+t)^{d+1}
\|u_0\|_{\M{p,q}{0,s}}+ C_1 t(1+t)^{d+1} \|u_1\|_{\M{p,q}{0, s-1}}
\end{equation}
where $C_{0}$ and $C_{1}$ are only functions of the dimension $ d$.

\end{proposition}

\begin{proof}
It was proved in \cite{bgor} that $\sigma_0(\xi)\in W(\cF L^1, L^\infty)$  and in
\cite{CNwave} that
$\sigma_1(\xi)\in W(\cF L^1,L^\infty_{1})$. In addition, it was shown in
\cite{CNwave}  that the solution satisfies
\begin{align*}
\|u(t,\cdot)\|_{\M{p,q}{0,s}}& \leq \|H_{\sigma_{0}}u_{0}\|_{\M{p,q}{0,s}} +
\|H_{\sigma_{1}}u_{1}\|_{\M{p,q}{0,s}}\\
& \leq \|H_{\sigma_{0}}u_{0}\|_{\M{p,q}{0,s}} +
\|H_{\sigma_{1}}u_{1}\|_{\M{p,q}{0,s-1}}\\
& \leq \|\sigma_{0}\|_{W(\cF L^1, L^\infty)} \|u_{0}\|_{\M{p,q}{0,s}} +
\|\sigma_{1}\|_{W(\cF L^1, L^\infty_{1})} \|u_{1}\|_{\M{p,q}{0,s-1}} \\
& \leq C_0 ( t) \|u_0\|_{\M{p,q}{0, s}}+ C_1(t) \|u_1\|_{\M{p,q}{0, s-1}}.
\label{es1A}
\end{align*}
We can now use the results proved in Section~\ref{main} to estimate $C_{0}(t)$ and
$C_{1}(t)$. More specifically, setting $\widetilde{\sigma_0}(\xi)=\cos  |\xi|$,  for
$t>0$, we can write
$\sigma_0(\xi)=(\widetilde{\sigma_0})_{{2\pi t}}$. Using~\eqref{mainbothW} with
$\mu_1(\infty,1)=1$, $\mu_2(\infty,1)=0$, we have, for every $R>0$,
\begin{equation*}
\|(\widetilde{\sigma_0})_{2\pi t}\|_{W(\cF L^1,
    L^\infty_{1})}\leq   \begin{cases} C_{0,R} \|\widetilde{\sigma_1}\|_{W(\cF L^1,
    L^\infty)} ,\quad t\leq R\\
 C'_{0,R} t^{d+1}\|\widetilde{\sigma_0}\|_{W(\cF L^1,
    L^\infty)}, \quad t\geq R.\end{cases}
 \end{equation*}
Hence $$C_0(t)\leq  \begin{cases} C_{0,R}, \,\quad 0\leq t\leq R\\
 C'_{0,R}t^{d+1},\,\quad t\geq R.
 \end{cases}
 $$
Setting $\widetilde{\sigma_1}(\xi)=\frac{\sin  |\xi|}{|\xi|}$,  for $t>0$, we can
write $\sigma_1(\xi)=t( \widetilde{\sigma_1})_{{2\pi t}}$ and, for every $R>0$,
\begin{equation*}
\|(\widetilde{\sigma_1})_{{2\pi t}}\|_{W(\cF L^1,
    L^\infty_{1})}\leq   \begin{cases} C_{1,R} \|\widetilde{\sigma_1}\|_{W(\cF L^1,
    L^\infty_{1})} ,\quad t\leq R\\
 C'_{1,R} t^{d+1}\|\widetilde{\sigma_1}\|_{W(\cF L^1,
    L^\infty_{1})}, \quad t\geq R.\end{cases}
 \end{equation*}
Hence $$C_1(t)\leq  \begin{cases} C_{1,R}t, \,\quad 0\leq t\leq R\\
 C'_{1,R}t^{d+2},\,\quad t\geq R,
 \end{cases}
 $$ and the estimate~\eqref{es1A} becomes
 \begin{equation*}\|u(t,\cdot)\|_{\M{p,q}{0, s}}\leq C_0 (1+t)^{d+1}
\|u_0\|_{\M{p,q}{0, s}}+ C_1t(1+t)^{d+1} \|u_1\|_{\M{p,q}{0,s-1}},
\quad t> 0.
\end{equation*}
\end{proof}

\subsubsection{Vibrating plate equation.}
 Consider now the following Cauchy problem for the vibrating plate equation
\begin{equation}\label{cpp}
\begin{cases}
\partial^2_t u+\Delta^2_x u=0\\
u(0,x)=u_0(x),\,\,
\partial_t u (0,x)=u_1(x),\,\,
\end{cases}
\end{equation}
with $t\geq 0$, $x\in\R^d$, $d\geq1$.  The formal solution  $u(t,x)$ is  given by
$$u(t,x)=\intrd e^{2\pi i x\xi} \cos(4\pi^2 t |\xi|^2)
\widehat{u_0}(\xi)\,d\xi+\intrd e^{2\pi i x\xi} \frac{\sin (4\pi^2 t
|\xi|^2)}{4\pi^2  |\xi|^2} \widehat{u_1}(\xi) \, d\xi,
$$ and satisfies the following estimate.

\begin{proposition} \label{L2} Let $s\in \R$, and $1\leq p, q \leq \infty$. Then,
the solution $u(t, x)$ of~\eqref{cpp}  with initial data $(u_{0}, u_{1}) \in
\M{p,q}{0,s}\times \M{p,q}{0, s-2}$
satisfies
\begin{equation}\label{es1B}\|u(t,\cdot)\|_{\M{p,q}{0, s}}\leq C_0 (1+ t)^{\frac d
2}
\|u_0\|_{\M{p,q}{0,s}}+ C_1 t(1+t)^{\frac d 2+1} \|u_1\|_{\M{p,q}{0, s-2}}
\end{equation}
where $C_{0}$ and $C_{1}$ are only functions of the dimension $ d$.
\end{proposition}

\begin{proof}
Here the solution is the sum of two Fourier multipliers $u=H_0 u_0+H_1 u_1$ having
symbols
$\sigma_0(\xi)= \cos(4\pi^2 t |\xi|^2) \in W(\cF L^1, L^\infty)$ (see \cite{bgor})
and $\sigma_1(\xi)=
\frac{\sin (4\pi^2 t |\xi|^2)}{4\pi^2  |\xi|^2} \in W(\cF L^1, L^\infty_2)$ (see
\cite{CZplate}).

Since $\sigma_0(\xi)=\cos( |\xi|^2)_{2\pi \sqrt{t}}$ and $\sigma_1(\xi)=t
\left(\frac{\sin ( |\xi|^2)}{ |\xi|^2}\right)_{2\pi \sqrt{t}}$, using  the same
arguments as for the wave equation we obtain:
$$\|u(t,\cdot)\|_{\M{p,q}{0, s}}\leq C_0 (1+ t)^{\frac d
2}\|u_0\|_{\M{p,q}{0, s}}+ C_1 t(1+t)^{\frac d 2+1}
\|u_1\|_{\M{p,q}{0, s-2}}, \quad t> 0.
$$
\end{proof}

\subsection{Embedding of Besov spaces into modulation spaces}\label{embed}
 We generalize some results of \cite{kasso04}. But first, we recall the inclusion
relations between Besov spaces and modulation spaces (see
\cite{sugimototomita,baoxiang3}). Consider the following indices, where $\mu_i$,
$i=1, 2$ were defined in Section~\ref{prelim}:
$$\nu_1(p,q)=\mu_1(p,q)+\frac1p,\quad\quad\nu_2(p,q)=\mu_2(p,q)+\frac1p.
$$

The following result was proved in \cite[Theorem 3.1]{toft04}
 and in \cite[Theorem 1.1]{baoxiang3}
\begin{theorem}\label{incl}
Let $1\leq p,q\leq \infty$ and $s\in\R$.

\noindent (i) If $s\geq
d\nu_1(p,q)$ then $
B^{p,q}_s(\rd)
\hookrightarrow
\M{p,q}{}(\rd)$.

\noindent (ii) If
$s\leq d\nu_2(p,q)$ then
$\M{p,q}{}(\rd)
\hookrightarrow
B^{p,q}_s(\rd)$.
\end{theorem}

The next results  improve those in \cite[Theorem 3.1]{kasso04}.
\begin{theorem}\label{teor31} Let $1\leq p\leq 2$.

\noindent (i)  If $s\geq d(1/p-1/p')$ and
$1\leq q\leq p$ then
$B^{p,q}_s\hookrightarrow \M{p}{}.
$

\noindent (ii) If $s> d(1/p-1/p')$ and $1\leq q\leq \infty$ then
$B^{p,q}_s\hookrightarrow \M{p}{}.
$

\end{theorem}

\begin{proof} (i) For  $s\geq d(1/p-1/p')\geq \nu_{1}(p,p)=0$, Theorem \ref{incl}
says that $B^{p,p}_s \hookrightarrow \M{p,p}{}$.
However, the inclusion relations for Besov spaces give $B^{p,q}_s\hookrightarrow
B^{p,p}_s$, for $q\leq p$. Hence the result follows. \\

(ii) If $s> d(1/p-1/p')\geq 0$, and $q\leq p$, then this is exactly (i) above. If
$p\leq q$, then $B^{p,q}_{s}\hookrightarrow B^{p,p} \hookrightarrow \M{p}{}.$
\end{proof}

The next results  improve those in \cite[Theorem 3.2]{kasso04}.

\begin{theorem}\label{teor32}

\noindent (i) Let $1\leq p\leq 2$, $s>0$. Then
$B^{p,q}_s\hookrightarrow \M{p,p'}{},\quad  \mbox{for\,all\,}\,\,1\leq q\leq\infty.
$

\noindent(ii) If  $2\leq p\leq \infty$, $s> d(1/p'-1/p)$, then
$B^{p,q}_s\hookrightarrow \M{p,p'}{},\quad  \mbox{for\,all\,}\,\,1\leq q\leq\infty.
$

\end{theorem}

\begin{proof} (i) For  $1\leq p\leq 2$, $\nu_1(p,p')=0$ and using  Theorem
\ref{incl} we obtain $B^{p,p'} \hookrightarrow \M{p,p'}{}$. Since $B^{p,q}_s
\hookrightarrow B^{p,p'}$, for all $1\leq q\leq\infty$, $s>0$, the result
follows.\\
(ii) If $2\leq p\leq \infty$, $$\nu_1(p,p')=\frac1{p'}-\frac1p\leq\frac1{p'}.$$
Hence, if $s\geq  d(1/p'-1/p)$, Theorem \ref{incl} gives $B^{p,p'}_s \hookrightarrow
\M{p,p'}{}$. If $s> d(1/p'-1/p)$,
the inclusion relations for Besov spaces give $B^{p,q}_s \hookrightarrow
B^{p,p'}_{d(1/p'-1/p)}$. This is easy to see if $q\leq p'$. On the other hand if
$q>p'$ it follows by an application of H\"older's inequality for $\ell^p$ spaces. In
any case, this concludes the proof.
\end{proof}

\section{Acknowledgment} The authors would like to thank Fabio Nicola for helpful
discussions. K.~A.~Okoudjou would also like to acknowledge the partial support of the
Alexander von Humboldt foundation.

\end{document}